\documentclass{article}

\usepackage{ulem}
\usepackage{amssymb,amsbsy,amsmath,amsfonts,amssymb,amscd,amsthm}
\usepackage{xcolor}
\usepackage{graphicx} 
\usepackage{multicol}
\graphicspath{{./}{figures/}} 
\usepackage{tikz}
\usetikzlibrary{matrix}

\newcommand{\norm}[1]{\|#1 \|}

\newcommand{\R}{\mathbb{R}}

\newcommand{\N}{\mathbb{N}}
\newcommand{\T}{{\mathbb{T}^d}}

\renewcommand{\H}{\mathcal{H}}

\newcommand{\beq}{\begin{equation}}
\newcommand{\eeq}{\end{equation}}

\def\D{\Delta}
\def\g{\gamma}

\def\e{\varepsilon}

\newcommand{\cG}{{\cal G}}

\newcommand{\diver}{{\rm div}}
\def\pd{\partial}

\newtheorem{theorem}{Theorem}[section]
\newtheorem{lemma}[theorem]{Lemma}

\newtheorem{remark}[theorem]{Remark}

\numberwithin{equation}{section}

\titlepage
\title{A policy iteration method for Mean Field Games}
\author{Simone Cacace, Fabio Camilli, Alessandro Goffi}
\date{\today}

\begin{document}

\maketitle

\begin{abstract} 
The policy iteration method is a classical algorithm for solving   optimal control problems.
In this paper, we introduce  a policy iteration method for    Mean Field Games  systems, and we study the convergence of this  procedure to a solution of the problem. We also introduce suitable discretizations to numerically solve both stationary and evolutive problems. We show  the convergence of the   policy iteration method for the discrete problem and we study the performance of the proposed algorithm on some examples in dimension one and two. 

\end{abstract}

\noindent
{\footnotesize \textbf{AMS-Subject Classification}}. {\footnotesize 49N70; 35Q91; 91A13; 49D10.}\\
{\footnotesize \textbf{Keywords}}. {\footnotesize Mean Field Games; policy iteration;  convergence; numerical methods.}


\section{Introduction}

Mean Field Games (MFGs) models have been introduced in \cite{hcm,ll} to describe stochastic differential games with a very large number of agents. They have  a wide range of applications in Engineering, Economics, and Finance \cite{cdl,gnp}.  From a mathematical point of view, MFGs theory leads to the study of  a  
system of   differential equations composed, in the finite horizon case, by a backward Hamilton-Jacobi-Bellman
(HJB) equation for the value function of the single agent  and  a Fokker-Planck (FP) equation governing the distribution of the overall population, i.e.
\begin{equation}\label{MFG}
\begin{cases}
-\partial_tu-\e \Delta u+H(Du)=F(m(x,t)) & \text{ in }Q\\
\partial_tm -\e \Delta  m-\diver(mD_pH(Du))=0 & \text{ in }Q\\
m(x,0)=m_0(x),\, u(x,T)=u_T(x) & \text{ in }\T\ ,
\end{cases}
\end{equation}
where $Q:=\T\times(0,T)$, $\T$ stands for the flat torus $\R^d / \mathbb{Z}^d$, $\e>0$, $H$ is the Hamiltonian term   and $F$ is the so-called coupling term,  depending  locally on the density.\par
Apart from some very specific cases such as the linear-quadratic one \cite{bp}, MFG  systems typically have no closed form solutions, hence they have to be solved numerically (see for example \cite{acd,accd,cs} and \cite{al1} for a review). The forward-backward structure of the system, the strong coupling among the equations and the nonlinearity of the HJB equation are important features of the MFG system, and various  strategies to solve the   finite-dimensional problems obtained via the discretization of the MFG system have been discussed in the literature (\cite{al,bcarda,bks,cc,ch}).\par
The policy iteration method is usually attributed to Bellman \cite{b} and Howard \cite{h} and it  has played a   pivotal role in the numerical solution  of deterministic and stochastic control problems, both in discrete and continuous settings. It can be interpreted as a linearization method for an intrinsically nonlinear problem, and its global convergence in the finite dimensional case  was proved in \cite{h}. Moreover,  Puterman and Brumelle \cite{pb} observed that the policy iteration method can be also seen as a Newton's algorithm for the nonlinear control problem; therefore, if the initial guess is in a neighborhood of the true solution, then the convergence is quadratic. For continuous control problems, assuming that the control set is bounded, the convergence of the method has been obtained by Fleming \cite{fl}  and Puterman \cite{pu1,pu2}, who used this procedure to give a constructive proof of the existence of classical and weak solutions to quasilinear parabolic differential equations arising in the control of non-degenerate diffusion processes. Instead, for deterministic control problems with continuous state space, despite the method is largely used in the computation of the value function and  the  optimal control, no general convergence result  is known.  For recent results about the policy iteration method and its applications, see \cite{alla,bmz,kss,santos}.\par
In this paper, we consider the following policy iteration algorithm for the MFG system \eqref{MFG}. Let $L(q)$ be the Lagrangian associated to the Hamiltonian $H$.
Fixed $R>0$ and given a bounded, measurable vector field   $q^{(0)}:\T\times [0,T]\to\R^d$  with $\|q^{(0)}\|_{L^\infty(Q)}\le R$, we iterate on $k\ge 0$
\begin{itemize}
	\item[(i)] Solve
	\begin{equation}\label{eq:alg_FP}
	\left\{
	\begin{array}{ll}
	\pd_t m^{(k)}-\e\Delta m^{(k)}-\diver (m^{(k)} q^{(k)})=0,\quad &\text{ in }Q\\
	m^{(k)}(x,0)=m_0(x)&\text{ in }\T.
	\end{array}
	\right.
	\end{equation}
	\item[(ii)] Solve
	\begin{equation}\label{eq:alg_HJ}
	\left\{
	\begin{array}{ll}
	-\pd_t u^{(k)}- \e\Delta u^{(k)}+q^{(k)}\cdot Du^{(k)}-L(q^{(k)})=F(m^{(k)})&\text{ in }Q\\
	u^{(k)}(x,T)=u_T(x)&\text{ in }\T.
	\end{array}
	\right.
	\end{equation}
	\item[(iii)] Update the policy
	\begin{equation}\label{eq:update_policy}
	q^{(k+1)}(x,t)={\arg\max}_{|q|\le R}\left\{q\cdot Du^{(k)}(x,t)-L(q)\right\}\quad\text{ in }Q.
	\end{equation}
\end{itemize}
At $k^{th}$-step, frozen the policy $q^{(k)}$, we first update   $m^{(k)}$ by means of the forward FP equation \eqref{eq:alg_FP},  we plug the new distribution of agents   in \eqref{eq:alg_HJ} computing
the corresponding value function $u^{(k)}$  and, lastly, we   determine the new  policy $q^{(k+1)}$ corresponding to the value function $u^{(k)}$. If the coupling cost $F$ is independent of the density $m$, step {\it (ii)} and {\it (iii)} of the previous algorithm coincide with the classical policy iteration method for the HJB equation in \eqref{MFG}.\par
In our first result, see Theorem \ref{thm:policy_iteration1},  we prove  convergence (up to a subsequence) of the policy iteration method for    the MFG system   \eqref{MFG} assuming that the Hamiltonian is convex and globally Lipschitz, hence in a setting similar to \cite{fl,pu1,pu2}.\\
Our second   result, see  Theorem \ref{thm:policy_iteration}, deals with Hamiltonians having polynomial growth and states that, for $R$ sufficiently large, the sequence $(u^{(k)},m^{(k)})$ given by \eqref{eq:alg_FP}-\eqref{eq:update_policy} converges (up to a subsequence)  to a   solution of \eqref{MFG}. Since this result  does not suppose  the existence of a   solution to \eqref{MFG} nor   monotonicity
assumptions,   it can be also seen as a constructive proof of the existence of solutions to \eqref{MFG}.\\
 As in \cite{fl,pu1,pu2}, our approach relies on a priori estimates for the solutions of the linear problems \eqref{eq:alg_FP}, \eqref{eq:alg_HJ} in spaces of maximal regularity and on compactness properties of the functional spaces where the solution of the (nonlinear) problem is defined. With respect to former papers, we have two additional difficulties: the method is applied to a system of PDEs instead that to a single equation; moreover, in the second result,   the Hamiltonian has polynomial gradient growth and therefore the control variable is defined in the whole $\R^d$. The latter point is solved by observing that, via an a priori gradient estimates from \cite{CG2} for the solution to the HJB equation obtained via duality arguments,  the behaviour of $H$ only matters in a sufficiently large ball $B(0,R)$. Hence, one can truncate the Hamiltonian,  note then that the solution of \eqref{MFG} and the one of the corresponding truncated problem     coincide  and, finally, solve via policy iteration method the latter problem to obtain an approximation of the former one. \par
We also briefly discuss in Section \ref{sec:stat} a treatment for the stationary counterpart of \eqref{MFG} introduced by Lasry and Lions \cite{ll}, i.e. we implement a policy iteration algorithm for the ergodic MFG system
\begin{equation*}
	\begin{cases}
	 -\e\Delta u+H(Du)+\lambda=F(m(x)) & \text{ in }\T\\
	 -\e\Delta  m-\diver(mD_pH(Du))=0 & \text{ in }\T\\
	\int_\T m(x)dx=1,\quad m\ge 0,\quad \int_\T u(x)dx=0 \ ,
	\end{cases}
	\end{equation*}
where $\lambda$ stands for the ergodic constant. As it is well-known, this system describes the long-time average asymptotics of solutions to \eqref{MFG} and it is widely analyzed in the literature, see e.g. \cite{nhm,pt} and the references therein. In this case, the convergence result for the policy iteration algorithm will be proved in Theorem \ref{thm:policy_iteration_stat}.

Finally, we introduce suitable discretizations for both stationary and evolutive MFGs, and we employ the policy iteration method to numerically solve the corresponding discrete systems. We show  the convergence of the  policy iteration method for the discrete problem and we explain   that it  can be interpreted as a quasi-Newton method applied to the discrete MFG system. Some numerical tests in dimension one and two complete the presentation, including a performance comparison with a full Newton method.     
\par 
\smallskip
 
The paper is organized as follows. In Section \ref{sec:prelim} we collect   definitions and some technical lemmas necessary to  prove  the convergence results for the parabolic problem, to which is devoted  Section \ref{sec:evol}. 
Section \ref{sec:stat} describes the policy iteration method for the stationary ergodic MFG system. Section \ref{sec:numapp} comprehends the numerical approximation and the convergence of the policy iteration for the discrete problem, while in Section  \ref{sec:test}  we show some tests.


\section{Notations and preliminary results}\label{sec:prelim}
In this section we introduce some functional spaces and state some preliminary results we need in the forthcoming sections.\\
We denote by $L^r(\T)$ the space of all measurable and periodic functions on $\R^d$ belonging to $L^r_{\mathrm{loc}}(\R^d)$ equipped with the norm $\|u\|_r=\|u\|_{L^r((0,1)^d)}$. For  $\mu\in(0,1)$, $r\geq1$,  we denote with $W^{\mu,r}(\T)$ the standard fractional Sobolev spaces of periodic functions $u\in L^r(\T)$ such that the semi-norm
\[
[u]_{W^{\mu,r}(\T)}=\left(\iint_{\T\times\T}\frac{|u(x)-u(y)|^r}{|x-y|^{d+\mu r}}\,dxdy\right)^{\frac1r}\ ,
\]
is finite, thus endowed with the natural norm $\|\cdot\|_{W^{\mu,r}(\T)}=\|\cdot\|_r+[\cdot]_{W^{\mu,r}(\T)}$.
When $\mu>1$ is non-integer, one writes $\mu=k+\sigma$, with $k\in\N$ and $\sigma\in(0,1)$ and $W^{\mu,r}(\T)$ comprehends those functions $f\in W^{k,r}(\T)$ (the standard integer-order Sobolev space on the torus) whose distributional derivatives $D^\alpha f$, $|\alpha|=k$, belong to $W^{\sigma,r}(\T)$ previously defined. We refer the reader to \cite{ST} for a treatment of fractional spaces on the torus as well as to \cite{CG1,Lun} for the definitions via real interpolation in Banach spaces, see also the references therein.\\
For any $r\geq1$, we denote by $W^{2,1}_r(Q)$ the space of functions $f$ such that $\partial_t^{\delta}D^{\beta}_x u\in L^r(Q)$ for all multi-indices $\beta$ and $\delta$ such that $|\beta|+2\delta\leq  2$, endowed with the norm
\begin{equation*}
\norm{u}_{W^{2,1}_r(Q)}=\left(\iint_{Q}\sum_{|\beta|+2\delta\leq2}|\partial_t^{\delta}D^{\beta}_x u|^rdxdt\right)^{\frac1r}.
\end{equation*}
We recall that, by classical results in interpolation theory, the sharp  space of initial (or terminal) traces of $W^{2,1}_r(Q)$ is given by the fractional Sobolev class $W^{2-\frac{2}{r},r}(\T)$, cf \cite[Corollary 1.14]{Lun}. To treat problems with divergence-type terms, we first define $W^{1,0}_s(Q)$ as the space of functions such that the norm
\[
\norm{u}_{W^{1,0}_s(Q)}:=\norm{u}_{L^s(Q)}+\sum_{|\beta|=1}\norm{D_x^{\beta}u}_{L^s(Q)}
\]
is finite. Then, we denote by $\H_s^{1}(Q)$ the space of those functions $u\in W^{1,0}_s(Q)$ with $\partial_t u\in (W^{1,0}_{s'}(Q))'$, equipped with the natural norm
\begin{equation*}
\norm{u}_{\mathcal{H}_s^{1}(Q)}:=\norm{u}_{W^{1,0}_s(Q)}
+\norm{\partial_tu}_{(W^{1,0}_{s'}(Q))'}\ .
\end{equation*}
For $\alpha\in(0,1)$, we denote the classical parabolic H\"older space $C^{\alpha,\frac{\alpha}{2}}(Q)$ as the space of functions $u\in C(Q)$ such that
\[
[u]_{C^{\alpha,\frac{\alpha}{2}}(Q)}:=\sup_{(x_1,t_1),(x_2,t_2)\in Q}\frac{|u(x_1,t_1)-u(x_2,t_2)|}{(\mathrm{dist}(x_1,x_2)^2+|t_1-t_2|)^{\frac{\alpha}{2}}}<\infty\ ,
\]
where $\mathrm{dist}(x,y)$ stands for the geodesic distance from $x$ to $y$ in $\T$. 
If   $s>d+2$, then $\H_s^1(Q)$ is continuously embedded onto $C^{\delta,\delta/2}(Q)$ for some $\delta\in (0,1)$, see \cite[Appendix A]{MPR}.\\
We now recall some standard parabolic regularity results we will use in the sequel.
\begin{lemma}\label{lemma:linear_FP}
	Let $g:Q\to\R^d$ be a bounded vector field   and  $m_0\in L^2(\T)$, $m_0\geq0$. Then the problem
	\begin{align*}
	\left\{
	\begin{array}{ll}
	\pd_t m- \e \Delta m+\diver (g(x,t)m)=0&\text{ in }Q,\\
	m(x,0)=m_0(x)&\text{ in }\T,
	\end{array}
	\right.
	\end{align*}
	has a unique   solution $m\in \H_2^1(Q)$, which is a.e. non negative on $Q$. Furthermore, if $m_0\in L^\infty(\T)$,  then $m\in L^\infty(Q)\cap \H_2^1(Q)$ and, if $m_0\in W^{1,s}(\T)$, $s\in(1,\infty)$, we have
	\begin{equation*}
	\|m\|_{\H_s^1(Q)}\le C
	\end{equation*}
	for some constant  $ C=C(\|g\|_{L^\infty(Q;\R^d)},\|m_0\|_{W^{1,s}(\T)})$.
\end{lemma}
\begin{proof}
The well-posedness and positivity of $m$ are standard matter that can be found in \cite{LSU}, while integrability estimates, even under weaker assumptions on the drift, can be found in \cite{BCCS}. When $m_0\in W^{1,s}(\T)$, the estimate in $\H_s^1$ can be obtained following the arguments in \cite[Proposition 2.2]{CG2}, although one can get the regularity result even when $m_0\in W^{1-2/s,s}(\T)$ via maximal regularity.
\end{proof}
\begin{lemma}\label{lemma:linear_HJ}
	Let $b\in L^\infty(Q;\R^d)$, $f\in L^r(Q)$ and $u_T\in W^{2-\frac{2}{r},r}(\T)$ for some $r>d+2$. Then the problem
	\begin{align*}
	\begin{cases}
	-\partial_t u-\e\Delta u+b(x,t) \cdot Du=f(x,t)&\text{ in }Q\\
	u(x,T)=u_T(x)&\text{ in }\T
	\end{cases}
	\end{align*}	
 admits a unique solution $u\in W^{2,1}_r(Q)$  and it holds
	\begin{equation}\label{max1}
	\|u\|_{W^{2,1}_r(Q)}\leq C(\|f\|_{L^r(Q)}+\|u_T\|_{W^{2-\frac{2}{r},r}(\T)}),
	\end{equation}
	where $C$ depends on the norm of the coefficients as well as on $r,d,T$ and remains bounded for bounded values of $T$. 
	Furthermore, we have $Du\in C^{\alpha,\alpha/2}$ for some $\alpha\in(0,1)$.\\
Finally, if the coefficients $b,f$ belong to $C^{\alpha,\alpha/2}(Q)$ and $u_T\in C^{2+\alpha}(\T)$, then
 	\begin{equation}\label{schauder}
	\|\partial_t u\|_{C^{\alpha,\frac{\alpha}{2}}(Q)}+\|D^2u\|_{C^{\alpha,\frac{\alpha}{2}}(Q)}\leq C(\|f\|_{C^{\alpha,\frac{\alpha}{2}}(Q)}+\|u_T\|_{C^{2+\alpha}(\T)})\,,
	\end{equation}	
	where $C$ depends on the $C^{\alpha,\alpha/2}$-norm of the coefficients as well as on $d,T$ and remains bounded for bounded values of $T$.
\end{lemma}
\begin{proof} 
	The estimate \eqref{max1} is a maximal regularity result that dates back to \cite[Theorem IV.9.1 p.342]{LSU}, obtained when $b\in L^r(Q;\R^d)$, $r>d+2$. The embedding of the spatial gradient in (parabolic) H\"older spaces for $r>d+2$ is proved in \cite[Corollary IV.9.1 p.342]{LSU}, see also the embeddings in \cite{CG1} (setting $s=1$) for a proof via a slightly different approach. \\
	The Schauder estimate \eqref{schauder} is proved in \cite[eq. (10.5), p. 352]{LSU}.
\end{proof}

\section{Convergence  of the policy iteration method: the evolutive problem}\label{sec:evol}

 In this section, we prove the convergence of the policy iteration method for the evolutive problem. 
Concerning the Hamiltonian, we focus on two different settings
\begin{itemize}
	\item[(i)] $H$ is differentiable, convex and globally Lipschitz  continuous, i.e.  there exists  a constant $\beta>0$ such that
	\begin{equation}\label{H1}
		|D_pH(p)|\leq \beta\qquad\text{ for all }p\in\R^d\ .
	\end{equation}
\item[(ii)] $H$ is of the form
\begin{equation}\label{H2}
	 H(p)=|p|^\gamma, \qquad \gamma>1.
\end{equation}
\end{itemize}
We define the Lagrangian $L:\R^d\to \R$ as the Legendre transform of $H$, i.e.
 $L(\nu)=\sup_{p\in\R^d}\left\{p\cdot \nu-H(p)\right\}$. In particular, it holds 
\begin{equation*}
H(p) = p\cdot q - L(q) \quad \text{if and only if} \quad q = D_p H(p)\ .
\end{equation*}
Note that, if \eqref{H1} holds, one can write  $H(p)=\sup_{|q|\le \beta}\{p\cdot q-L(q)\}$ and therefore in this case we may assume that the set of controls is bounded.\\
Concerning the coupling cost, we consider bounded local couplings by assuming that $F:\R^+\to\R$ is a continuous,  uniformly bounded  function, i.e. there exists $C_F>0$ such
\begin{equation}\label{F}
|F(m)|<C_F\quad \text{ for }m\geq0.
\end{equation}
 Finally, we suppose   that
\begin{equation}
\label{I}
\begin{aligned}
&\text{$u_T\in W^{2-\frac{2}{r},r}(\T)$, $r>d+2$,}\\ 
&\text{$m_0\in W^{1,s}(\T)$, $s>d+2$, is non-negative and $\int_{\T}m_0(x)dx=1$}.
\end{aligned}
\end{equation}

Our first result concerns the case of a globally Lipschitz Hamiltonian, and extends to  MFG systems the works by Fleming and Puterman.
\begin{theorem}\label{thm:policy_iteration1}
	Let \eqref{H1},  \eqref{F}, \eqref{I} be  in force. Then the sequence $(u^{(k)},m^{(k)})$, generated by the policy iteration algorithm converges,  up to a subsequence,   to a (strong) solution  $(u ,m)\in W^{2,1}_r(Q)\times   \H_s^{1}(Q)$  of \eqref{MFG}.\\
	Moreover, if
	\begin{equation} \label{unique}
	\int_{\T}(F(m_1)-F(m_2))d(m_1-m_2)(x)>0\ ,\forall m_1,m_2\in\mathcal{P}(\T)\ ,m_1\neq m_2\ ,
	\end{equation}
	 then all the sequence converges to the unique solution of \eqref{MFG}.
\end{theorem}
\begin{proof}
Due to assumption \eqref{H1}, we have  
	\begin{equation}\label{eq:H_bounded}
	H(p)=\sup_{|q|\le R}\left\{p\cdot q-L(q)\right\}.
	\end{equation}	
for $R=\beta$. Moreover, the drift  in  the Fokker-Planck equation is uniformly bounded (independently of $u$). 
Given the vector field $q^{(k)}$ defined as in \eqref{eq:update_policy} at  step $k-1$, by Lemma \ref{lemma:linear_FP}, in view of the boundedness of the velocity field,
we infer the existence of a unique weak solution $m^{(k)}$ of \eqref{eq:alg_FP} satisfying
	\begin{equation}\label{eq:conv_est1}
	\|m^{(k)}\|_{\H_s^1(Q)}\le C. 
	\end{equation}
Moreover, by    Lemma \ref{lemma:linear_HJ} and \eqref{F}, 	there exists a unique strong  solution $u^{(k)}\in W^{2,1}_r(Q)$ such that
	\begin{equation}\label{eq:conv_est2a}
	\|u^{(k)}\|_{W^{2,1}_r(Q)}\le C(\|F(m^{(k)})\|_{L^r(Q)}+\|u_T\|_{W^{2-\frac{2}{r},r}(Q)})\,,
	\end{equation}
with $C$ depending only on $\beta$. Since $r>d+2$, by parabolic Sobolev embeddings we have
	\begin{equation}\label{eq:conv_est2b}
	\|Du^{(k)}\|_{C^{\alpha,\frac{\alpha}{2}}(Q)}\leq C
	\end{equation}
and,  by the hypotheses on $H$, this implies that $H(Du^{(k)})$ is space-time H\"older continuous. 
Since the supremum in \eqref{eq:H_bounded} is attained at $D_pH(Du^{(k)})$, we have that 
	\begin{equation*}
	q^{(k+1)}(x,t)=\mathrm{argmax}_{|q|\le R} \left\{q\cdot Du^{(k)}(x,t)-L(q)\right\}=D_pH(Du^{(k)}).
	\end{equation*}
In view of \eqref{eq:conv_est1} and the continuous embedding of $\H^1_s(Q)$ in $C^{\delta,\frac{\delta}{2}}(Q)$ for some $\delta\in (0,1)$, then there exists a subsequence, still denoted by $m^{(k)}$, which  uniformly converges to a continuous function $m$. 
By \eqref{eq:conv_est2a}, \eqref{eq:conv_est2b}, there exists a subsequence, still denoted by $u^{(k)}$, and a function $u$ such that $u^{(k)},Du^{(k)}$  converge uniformly to $u,Du$ and $\pd_t u^{(k)}, D^2u^{(k)}$ converge weakly in $L^r(Q)$ to $\partial_tu, D^2u$.\\
Consider the subsequence $(u^{(k)},m^{(k)})$ obtained by first  extracting a subsequence $m^{(k)}$  converging to $m$ and then a subsequence  $u^{(k)}$ converging to $u$. 
Then, passing to the limit in the weak formulation of \eqref{MFG} by means of the aforementioned convergences, one finds that the limit value $(u,m)\in W^{2,1}_r(Q)\times\H_s^1(Q)$ is a solution to \eqref{MFG} in distributional sense. \\
Finally, if assumption \eqref{unique} holds,   by a classical argument in \cite{ll}, (see  \cite[Theorem 5.1]{CG1} with $s=1$ and observe that it is only necessary to have sufficiently smooth solutions to run the arguments), the system \eqref{MFG} has a unique solution $(u,m)$. Hence, since any converging subsequence of the policy iteration method converges to the same limit, we get that all the sequence $(u^{(k)},m^{(k)})$ converges to the unique solution of \eqref{MFG}.
\end{proof}
In the proof of the previous result,  we also obtain the uniform convergence of the policy $q^{(k)}=D_pH(Du^{(k-1)})$ to  the optimal control for the limit problem $q=D_pH(Du)$.\\
We now consider the case of a Hamiltonian of the form $H(p)=|p|^\gamma$, for $\gamma>1$. Our second main result is the following
\begin{theorem}\label{thm:policy_iteration}
	Let  \eqref{H2}, \eqref{F}, \eqref{I} be  in force. Then, for $R$ sufficiently large in \eqref{eq:update_policy}, the sequence $(u^{(k)},m^{(k)})$, generated by the policy iteration algorithm converges,  up to a subsequence,   to a solution  $(u ,m)\in W^{2,1}_r(Q)\times   \H_s^{1}(Q)$  of \eqref{MFG}. Moreover, if \eqref{unique} holds, then all the sequence converges to the unique solution of \eqref{MFG}.
\end{theorem}
In this case, the main ingredient of the proof is an  a priori gradient estimate recently obtained in \cite{CG2}
  for strong solutions to Hamilton-Jacobi equations with $H$ as in \eqref{H2} and unbounded right-hand sides, which is stated in the next lemma for bounded source terms.
\begin{lemma}\label{bern}
Let $f\in L^\infty(Q)$, $H$ differentiable
and, for some $\g>1$,
\begin{equation*}
	D_pH(p)\cdot p-H(p)\ge c_H |p|^\gamma-\tilde c_H
\end{equation*}
\begin{equation*}
C_H^{-1}|p|^\gamma-C_H\leq	H(p)\leq C_H(|p|^\gamma+1)
\end{equation*}
\begin{equation*}
C_H^{-1}|p|^{\gamma-1}-C_H\leq	|D_pH(p)|\leq C_H(|p|^{\gamma-1}+1)
\end{equation*}
for all $p\in\R^d$ and some positive constants $c_H,\tilde c_H,C_H$. Then, if $u\in W^{2,1}_r(Q)$ is a strong solution to
\begin{equation*}
\partial_t u-\e\Delta u+H(Du)=f(x,t)\quad\text{ in }Q\ 
\end{equation*}
then  there exists a constant $C$ depending only on the data and $c_H,\tilde c_H,C_H$ such that
\begin{equation}\label{Lip_est}
	\|Du\|_{L^\infty(Q)}\leq C.
\end{equation}
\end{lemma}
\begin{proof}
The proof of this result, based on the Bernstein gradient estimate and the nonlinear adjoint method, can be found in \cite[Theorem 1.3]{CG2}, noting that, since $f\in L^\infty(Q)$, then $f\in L^q$ for every $q>1$ and that assumption \eqref{I} implies $u_T\in W^{1,\infty}(\T)$ by standard Sobolev embeddings. We emphasize that the Bernstein procedure can be applied to strong solutions in $W^{2,1}_r$ arguing as in \cite{CG4} without need to differentiate the equation.
\end{proof}

We are now in position to prove Theorem \ref{thm:policy_iteration}. In the proof, we first introduce a truncated Hamiltonian, which is globally Lipschitz continuous. This is due to the fact that the solution of the first equation in \eqref{MFG} satisfies the gradient bound in Lemma \ref{bern}, which readily implies that the behaviour of $H$ is important merely for $p\in B(0,R)$, $R\sim \|Du\|_{L^\infty}$ and therefore, for $R$ large enough in \eqref{eq:update_policy}, a solution of the truncated problem is also a solution of the original one. As a result, we can apply the convergence result proved in Theorem \ref{thm:policy_iteration1} to the MFG system with $H$ given by \eqref{H2}.
\begin{proof}[Proof of Theorem \ref{thm:policy_iteration}]
Owing to the bound \eqref{Lip_est} and following e.g. \cite{al}, one introduces a truncated Hamiltonian defined as
\begin{equation*}
	H_S(p)=\begin{cases}
		|p|^\gamma&\text{ if }|p|< S\ ,\\
		(1-\gamma)S^\gamma+\gamma S^{\gamma-1}|p|&\text{ if }|p|\geq S\ ,
	\end{cases}	
\end{equation*}
and the problem
\begin{equation}\label{MFGS}
	\begin{cases}
		-\partial_tu -\e\Delta u +H_S(Du )=F(m (x,t)) & \text{ in }Q,\\
		\partial_tm  -\e\Delta  m -\diver(m D_pH_S(Du ))=0 & \text{ in }Q,\\
		m(x,0)=m_0(x),\, u(x,T)=u_T(x) & \text{ in }\T\ .
	\end{cases}
	\end{equation}
We  observe that $H_S$ satisfies
\[
D_pH_S(p)\cdot p-H_S(p)=\begin{cases}
		|p|^\gamma&\text{ if }|p|< S\ ,\\
		(\gamma-1)S^{\gamma-1}&\text{ if }|p|\geq S\ .
	\end{cases}
\]
Given a solution $(u_S,m_S)$ of the system \eqref{MFGS}, repeating the same proof of \cite{CG2}, one first proves the bound
\[
\iint_Q \rho_S\min\{|Du|^\gamma,S^\gamma\}\,dxdt\leq C
\]
with $C$ independent of $S$, being $\rho_S$ the solution to the adjoint problem
\[
\begin{cases}
		\partial_t\rho  -\e\Delta  \rho-\diver(\rho D_pH_S(Du_S))=0 & \text{ in }Q,\\
		\rho(x,\tau)= \rho_\tau(x) & \text{ in }\T\ ,
	\end{cases}
\]
where $\rho_\tau\in C^\infty(\T)$ with $\|\rho_\tau\|_1=1$. Then, using the previous estimate, one gets  the bound on $\|D u_S (\cdot, \tau)\|_{L^\infty(\T)}$,  independent of $S$.  So, if we take $S$ large enough, we have that a solution of \eqref{MFGS} is also a solution of \eqref{MFG}. Finally, since $H_S$ is globally Lipschitz continuous,   the convergence of the policy iteration method to a solution of \eqref{MFGS}, and therefore of \eqref{MFG}, follows by Theorem \ref{thm:policy_iteration1}.
\end{proof}
Some comments on the previous result and its proof  are in order. 
\begin{remark}
	For a Hamiltonian satisfying \eqref{H2},
	  Theorem \ref{thm:policy_iteration}   gives a convergence result for a policy iteration method obtained by truncating the Hamiltonian at each step,    not the original problem. On the other hand, from the point of view of the numerical resolution of the problem, the truncation of the control space is natural since the   calculation of the optimal control must be performed on a bounded domain.
\end{remark}
\begin{remark}
  In the case of a regularizing coupling $F$  and regular final data $u_T\in C^{2+\alpha}(\T)$, the    convergence results for the policy iteration method  in Theorems \ref{thm:policy_iteration1} and \ref{thm:policy_iteration} hold in the space $C^{2+\alpha,1+\frac{\alpha}{2}}(Q)\times   \H_s^{1}(Q)$. Indeed, in this case, it is possible to exploit the Schauder-type estimate in Lemma \ref{lemma:linear_HJ} since $Du\in C^{\alpha,\alpha/2}(Q)$ by parabolic Sobolev embeddings. Therefore  the linear HJ equation can be regarded as a linear problem with space-time H\"older coefficients.	  
\end{remark}

\section{The ergodic problem}\label{sec:stat}
	We   consider  the stationary MFG system
	\begin{equation}\label{MFG_stat}
	\begin{cases}
	 -\e\Delta u+H(Du)+\lambda=F(m(x)) & \text{ in }\T\\
	 -\e\Delta  m-\diver(mD_pH(Du))=0 & \text{ in }\T\\
	\int_\T m(x)dx=1,\quad m\ge 0,\quad \int_\T u(x)dx=0 \ .
	\end{cases}
	\end{equation}
	For fixed $R>0$ and given a bounded, measurable function $q^{(0)}$ such that $\|q^{(0)}\|_{L^\infty(\T)}\le R$,
	a policy iteration method for \eqref{MFG_stat} is given by
    \begin{itemize}
    	\item[(i)]  Solve
    	\begin{equation}\label{eq:alg_FP_stat}
    	\left\{
    	\begin{array}{ll}
    	-\e\Delta m^{(k)}-\diver (m^{(k)} q^{(k)})=0,\quad &\text{ in }\T\\
    	\int_\T m^{(k)}(x)dx=1,\quad m^{(k)}\ge 0.
    	\end{array}
    	\right.
    	\end{equation}
    	\item[(ii)] Solve
    	\begin{equation}\label{eq:alg_HJ_stat}
    	\left\{
    	\begin{array}{ll}
    	 -\e\Delta u^{(k)}+q^{(k)}\cdot  Du^{(k)}-L(q^{(k)})+\lambda^{(k)}=F(m^{(k)}(x))&\text{ in }\T\\
    	\int_\T u^{(k)}(x)dx=0.
    	\end{array}
    	\right.
    	\end{equation}
    	\item[(iii)] Update the policy
    	\begin{equation}\label{eq:update_policy_stat}
    	q^{(k+1)}(x,t)={\arg\max}_{|q|\le R}\left\{q\cdot Du^{(k)}(x)-L(q)\right\}\quad\text{ in }\T.
    	\end{equation}
    \end{itemize}
 We have the following convergence theorems for the stationary case.
 
  \begin{theorem}
 	Let \eqref{H1} and \eqref{F}  be  in force. Then, the sequence $(u^{(k)},\lambda^{(k)}, m^{(k)})$, generated by the policy iteration algorithm converges,  up to a subsequence,   to a solution  $(u,\lambda ,m)\in W^{2,r}(\T)\times\R\times   W^{1,s}(\T)$  of \eqref{MFG_stat}, uniformly in $\T$. Moreover, if \eqref{unique} holds,
 	then all the sequence converges to the unique solution of \eqref{MFG_stat}.
 \end{theorem}   
 \begin{proof}
The proof goes along the same lines as Theorem \ref{thm:policy_iteration1} and  we omit it. The only difference relies on the use of the regularity results  given in Lemma  2 and 3 in \cite{bf}, which are the stationary counterpart of     Lemma \ref{lemma:linear_FP} and Lemma \ref{lemma:linear_HJ}.
 \end{proof}
 \begin{theorem}\label{thm:policy_iteration_stat}
 	Let \eqref{H2} and  \eqref{F}  be  in force. Then, for $R$ sufficiently large, the sequence $(u^{(k)},\lambda^{(k)}, m^{(k)})$, generated by the policy iteration algorithm converges,  up to a subsequence,   to a solution  $(u,\lambda ,m)\in W^{2,r}(\T)\times\R\times   W^{1,s}(\T)$  of \eqref{MFG_stat}, uniformly in $\T$. Moreover, if \eqref{unique} holds,
 	then all the sequence converges to the unique solution of \eqref{MFG_stat}.
 \end{theorem}   
\begin{proof}
The proof is similar to the one of the parabolic case and we do not give the details. To obtain the stationary analogue of Lemma \ref{bern} it is enough to adapt the proof in \cite{CG2} considering the dual equation
\[
-\epsilon \Delta \rho+\rho-\mathrm{div}(D_pH(Du)\rho)=\psi\qquad\text{ in }\T\ ,
\]
where $\psi\in C^\infty(\T)$, $\|\psi\|_1=1$ plays the same role of the initial datum of the parabolic adjoint problem in \cite{CG2}. As alternative, one can use the integral Bernstein gradient estimate in \cite[Theorem III.1]{lions85} (see also \cite{CG4}) as a counterpart of that in Lemma \ref{bern}.
\end{proof}

\section{Numerical approximation}\label{sec:numapp}
In this section, we present some details on the numerical approximation of the stationary/evolutive MFG systems, and we prove the convergence of 
the corresponding discrete  policy iteration method in a simple setting. 
We consider the reference case of the Eikonal-diffusion HJB equation, namely we choose the Hamiltonian
\begin{equation*}
	H(x,Du)=\frac12 |Du|^2-V(x)=\sup_{q\in\R^d}\left\{q\cdot Du-\frac12 |q|^2 -V(x)\right\}\,,
\end{equation*}
where $V$ is a given bounded  potential, and 
we focus on 
the  stationary ergodic problem \eqref{MFG_stat}. 

We define a grid $\mathcal{G}$ on $\T$, the vectors $U,M$ approximating respectively $u,m$ at the grid nodes, and the number $\Lambda$ approximating the ergodic cost $\lambda$. Then, we approximate \eqref{MFG_stat} by the following nonlinear problem on $\mathcal{G}$, 
\begin{equation}\label{MFG_stat_eik_lap_discrete}
	\begin{cases}
		-\varepsilon\D_\sharp U+\frac12 |D_\sharp U|^2+\Lambda=V_\sharp+F_\sharp(M) & \\
		-\varepsilon\D_\sharp M-\diver_\sharp(M\,D_\sharp U)=0 & \\
		\int_\sharp M =1\,,\quad M\ge 0\,,\quad \int_\sharp U =0&
	\end{cases}
\end{equation}
where, in order to avoid cumbersome notation, we use the symbol $\sharp$ to denote suitable discretizations of the linear differential operators, evaluations of functions at the grid nodes, and quadrature rules for the integrals. Typical choices on uniform grids are centered second order finite differences for the discrete Laplacian, and simple rectangular quadrature rules for the integral terms, whereas the Hamiltonian and the divergence term in the FP equation are both computed via the Engquist-Osher numerical flux for conservation laws. For instance, in dimension $d=1$, given a uniform discretization of $\T$ with $I$ nodes $x_i$, for $i=0,\dots,I-1$, and space step $h=1/I$, we have
$$(\D_\sharp U)_i=\frac{1}{h^2}\left(U_{[i-1]}-2U_i+U_{[i+1]}\right)\,,$$
$$
(D_\sharp U)_i=\left(D_L U_i\,,\,D_R U_i\right)=\frac{1}{h}\left( U_i-U_{[i-1]}\,,\,U_{[i+1]}-U_i\right)\,,
$$
where the index operator $[\cdot]=\left\{(\cdot +I)\,mod\, I\right\}$ accounts for the periodic boundary conditions. Moreover, using the notation $(\cdot)^+=\max\left\{\cdot,0\right\}$ and $(\cdot)^-=\min\left\{\cdot,0\right\}$, we have     
$$
(|D_\sharp U|^2)_i=\left( D_L U_i^+\right)^2+\left(D_R U_i^-\right)^2\,,
$$
\begin{align*}
	\left(\diver_\sharp(M\,D_\sharp U)\right)_i=&\frac{1}{h}
	\left( M_{[i+1]} D_L U_{[i+1]}^+ - M_i D_L U_i^+ \right)\\
	+& \frac{1}{h}\left(
	M_i D_R U_i^- -M_{[i-1]} D_R U_{[i-1]}^-
	\right)\,,
\end{align*}
$$
(F_\sharp(M))_i=F(M_i)\,,\quad (V_\sharp)_i=V(x_i)\,,\quad \int_\sharp M=h\sum_{i=0}^{I-1}M_i\,,\quad \int_\sharp U=h\sum_{i=0}^{I-1}U_i\,,
$$
  We refer the reader to \cite{acd,al1,cc} and the references therein, for further details on the discretization of MFG systems. 

It is well known that the two-sided gradient $D_\sharp$ is designed to approximate viscosity solutions to Hamilton-Jacobi equations, and to correctly catch, for first order equations, possible kinks in the solution $U$. 
It is worth noting that, at a formal level, $D_\sharp U$ acts in the scheme as a vector field with a number of components $2d$, doubled with respect to dimension $d$ of the problem. This suggests a natural way to approximate the policy $q$ in \eqref{eq:alg_FP_stat}-\eqref{eq:alg_HJ_stat} when building the policy iteration algorithm. Indeed, 
given an initial guess  
$Q=(Q_L,Q_R):\mathcal{G}\to\R^{2d}$ and using the notation $Q_\pm=(Q_L^+,Q_R^-)$, we set 
$Q^{(0)}=Q$ and we iterate on $k\ge 0$ the following steps: 

\begin{itemize}
	\item[(i)]  Solve
	\begin{equation*}
	\left\{
	\begin{array}{ll}
		-\varepsilon\D_\sharp M^{(k)}-\diver_\sharp(M^{(k)}\,Q^{(k)})=0
		,
		\quad &\text{ on }\mathcal{G}\\
		\int_\sharp M^{(k)} =1\,,\quad M^{(k)}\ge 0.
	\end{array}
	\right.
	\end{equation*}
	\item[(ii)] Solve
	$$
	\left\{
	\begin{array}{ll}
		-\varepsilon\D_\sharp U^{(k)}+Q^{(k)}_\pm\cdot D_\sharp U^{(k)}+\Lambda^{(k)}=\frac12 |Q^{(k)}_\pm|^2+V_\sharp+F_\sharp(M^{(k)})&
		\text{ on }\mathcal{G}\\
		\int_\sharp U^{(k)} =0\,.
	\end{array}
	\right.
	$$
	\item[(iii)] Update the policy
	\begin{equation}\label{eq:update_policy_discrete}
		Q^{(k+1)}=
	\begin{cases}
		D_\sharp U^{(k)}\quad &\mbox{if $|D_\sharp U^{(k)}|\le R$}\\[4pt]
		\frac{D_\sharp U^{(k)}}{|D_\sharp U^{(k)}|} R&\mbox{if $|D_\sharp U^{(k)}|> R$} 				
	\end{cases} 
	\quad\text{ on }\mathcal{G}.
	\end{equation} 
\end{itemize}

In the following theorem, we prove the convergence of the above discrete policy iteration, in the case of a quadratic Hamiltonian and in dimension one, 
but the argument can be  extended with  similar techniques to any dimension and   more general Hamiltonians.
\begin{theorem}\label{PI_dis}
	For 	$R$ in \eqref{eq:update_policy_discrete}  sufficiently large,  the sequence $(U^{(k)},\Lambda^{(k)}, M^{(k)})$, generated by the policy iteration algorithm, converges,  up to a subsequence,   to a solution  $(U,\Lambda,M)$ of \eqref{MFG_stat_eik_lap_discrete}. Moreover, if \eqref{unique} holds, 	 then all the sequence converges to the unique solution of  \eqref{MFG_stat_eik_lap_discrete}.
\end{theorem}
\begin{proof}  
	We first show that the policy iteration algorithm is well defined.
To this end, we begin with the discrete FP equation in (i), namely we consider the following matrix
\begin{equation*}
	A(Q):=	-\varepsilon\D_\sharp  -\diver_\sharp(\cdot\,Q)\,,
\end{equation*}
for a given $Q=(Q_L,Q_R):\mathcal{G}\to\R^2$ such that $|Q|\le R$ for some $R>0$. 
We claim that $A(Q)$ is singular (e.g., for $Q=0$, we simply get the discrete Laplacian with periodic boundary conditions, which has a zero eigenvalue). More precisely, we show that $dim(ker(A(Q)))=1$. For $i=0,\dots,I-1$, the non zero entries of $A(Q)$ are the following
$$
\begin{array}{c}
A_{i,i}(Q)=2\displaystyle\frac{\varepsilon}{h^2}+\frac{1}{h}Q_{L,i}^+-\frac{1}{h}Q_{R,i}^-\\
A_{i,[i-1]}(Q)=\displaystyle-\frac{\varepsilon}{h^2}+\frac{1}{h}Q_{R,[i-1]}^-\qquad
A_{i,[i+1]}(Q)=\displaystyle-\frac{\varepsilon}{h^2}-\frac{1}{h}Q_{L,[i+1]}^+\,.
\end{array}
$$
By Fredholm alternative, $dim(ker(A(Q)))=dim(ker(A^T(Q)))$, 
where the transpose matrix has the following non zero entries
$$
\begin{array}{c}
A_{i,i}^T(Q)=2\displaystyle\frac{\varepsilon}{h^2}+\frac{1}{h}Q_{L,i}^+-\frac{1}{h}Q_{R,i}^-\\
A_{i,[i-1]}^T(Q)=\displaystyle-\frac{\varepsilon}{h^2}-\frac{1}{h}Q_{L,i}^+\qquad
A_{i,[i+1]}^T(Q)=\displaystyle-\frac{\varepsilon}{h^2}+\frac{1}{h}Q_{R,i}^-
\end{array}\,,
$$
namely
$$
A^T(Q)=	-\varepsilon\D_\sharp  +Q_L^+ D_L+ Q_R^- D_R=-\varepsilon\D_\sharp+Q_\pm\cdot D_\sharp\,,
$$
which is exactly the linear operator appearing in the linearized HJ equation (ii) (conversely this duality is just exploited in \cite[Remark 1]{acd} to define the discrete divergence operator $\diver_\sharp$).
Since the Hamiltonian
\begin{equation}\label{eq:discrete_Ham}	
		g(x_i,p_1,p_2)=\frac 1 h Q^{+}_{L,i} p_1+\frac 1 h Q^{-}_{R,i} p_2,\qquad (p_1,p_2)\in \R^2,
\end{equation}
is continuous, non decreasing with respect to $p_1$, non increasing with respect to $p_2$ and  convex, it can be proved that the equation $A^T(Q)U=0$ admits only constant solutions $(C,\dots,C)\in\R^{|\cG|}$ for $C\in\R$, see step 1 of \cite[Theorem 1]{acd}. We conclude that $dim(ker(A^T(Q)))=1$ and the claim is proved.

We now build a solution $M\in\R^{|\cG|}$ of the discrete FP equation $A(Q)M=0$ satisfying $M\ge 0$ and $\int_\sharp M=1$. To this hand, we recall that $|Q|\le R$ and we observe that $A(Q)$ has a non negative diagonal and non positive off-diagonals, since by definition $Q_L^+\ge 0$ and $Q_R^-\le 0$. This implies that, for $\mu>0$ sufficiently large, the matrix $\mu I
+A(Q)$ is a non singular M-matrix, hence the following iterations on $s\ge 0$ are well defined
$$	(\mu I +A(Q) ) W^{(s+1)} = \mu W^{(s)}\,.$$
Moreover, if we choose $W^{(0)}\in\R^{|\cG|}\setminus \{0\}$ such that $W^{(0)}\ge 0$ and $\int_\sharp W^{(0)}=1$, the same properties hold for every $s\ge 0$, respectively due to the monotonicity of the M-matrix and by definition of $A(Q)$. In particular, the sequence $W^{(s)}$ is bounded. Therefore it converges, 
up to a subsequence, to a vector $M\ge 0$ satisfying $\int_\sharp M=1$
and 
$$	(\mu I +A(Q) ) M = \mu M\qquad \Longleftrightarrow \qquad A(Q)M=0\,.$$
Since $dim(ker(A(Q)))=1$, it follows that $M$ is the unique solution of the discrete FP equation satisfying $\int_\sharp M=1$. In particular, the whole sequence $W^{(s)}$ is convergent and it provides a constructive way to compute $M$. Surprisingly, we found out that the condition $\int_\sharp M=1$ is enough to prevent a change of sign in $M$, hence, a posteriori, the condition $M\ge 0$ can be omitted.\\
Summing up, we proved that, for every $Q^{(k)}:\mathcal{G}\to\R^2$ such that $|Q^{(k)}|\le R$, there exists a unique solution to the problem in step (i) of the policy iteration.


	We now consider the problem in  step (ii). Using again the Hamiltonian  \eqref{eq:discrete_Ham} with  $Q^{(k)}$ in place of $Q$, it can be proved (as in step 1 of \cite[Theorem 1]{acd}) 
	that the problem admits a unique solution $(U^{(k)},\Lambda^{(k)})$ satisfying the normalization condition  $\int_\sharp U^{(k)} =0$.  Moreover the following estimates hold
	\[
	|\Lambda^{(k)}|\le C_1,\quad \max_{\cG}(|D_\sharp U^{(k)}|)\le C_2,
	\]
	for two positive constants $C_1$, $C_2$ depending on $R$, 
	$\max_{\cG}|V_\sharp|$ and $\max_{\cG}|F_\sharp|$. Hence,  also  the sequence $\{(U^{(k)},\Lambda^{(k)})\}_{k\in\N}$ is bounded  in $\R^{|\cG|}\times\R$.\\
	Therefore,
	up to   a subsequence, we find that, as $k\to \infty$,   $(U^{(k)},\Lambda^{(k)}, M^{(k)})$ converges to  $(U,\Lambda,M)\in \R^{|\cG|}\times \R\times \R^{|\cG|}$ and $Q^{(k)}$ converges to $Q\in \R^{|\cG|\times |\cG|}$.
	Moreover, passing to the limit in (i)-(iii), $(U,\Lambda,M)$ satisfies 
	\begin{equation}\label{eq:MFG_discrete_truncated}
		\left\{
		\begin{array}{ll}
			-\varepsilon\D_\sharp U +Q \cdot D_\sharp U+\Lambda=\frac12 |Q|^2+V_\sharp+F_\sharp(M)&\\
			-\varepsilon\D_\sharp M-\diver_\sharp(MQ)=0 & \\
			\int_\sharp M =1\,,\quad M\ge 0\,,\quad \int_\sharp U =0&
		\end{array}
		\right.
	\end{equation}
	and
	\begin{equation*}	
		Q =  
		\begin{cases}
			D_\sharp U\quad &\mbox{if $|D_\sharp U|\le R$},\\[4pt]
			\frac{D_\sharp U}{|D_\sharp U|} R&\mbox{if $|D_\sharp U|> R$} . 				
		\end{cases}  
	\end{equation*}
	By \cite[Theorem 3]{acd}, \eqref{MFG_stat_eik_lap_discrete} has a  solution and, since the problem is discrete, it trivially satisfies  
	\begin{equation}\label{eq:stima_grad_discrete}	
		\max_{\cG}|D_\sharp U|\le C ,
	\end{equation}
	for some constant $C$, depending on $h$.
	Hence, for $R$ sufficiently large,  solutions to \eqref{eq:MFG_discrete_truncated}  are also solutions to \eqref{MFG_stat_eik_lap_discrete}    and therefore, for such $R$, we get the convergence of the policy iteration method. Moreover, if \eqref{unique} holds, then the solution of \eqref{MFG_stat_eik_lap_discrete} is unique (see\cite[. 3]{acd}) and therefore we get the convergence of all the sequence.
\end{proof}    	
\begin{remark}   
	As observed, the estimate \eqref{eq:stima_grad_discrete}	 is not uniform in $h$. But, since we are studying
	the convergence of the policy iteration method for $h$ fixed, this is not an issue at this level. Estimates on the discrete gradient, uniform in $h$, are provided in \cite{acd, accd}, also for more general Hamiltonians. They are important to study the convergence of the discrete problem to the continuous one, but we do not consider this point here.
\end{remark}	

For the sake of comparison, we consider here  the direct method for stationary MFGs introduced in \cite{cc}, which is based on a Newton-like algorithm applied to the full system \eqref{MFG_stat_eik_lap_discrete}, rewritten as a multidimensional root-finding problem. More precisely, performing a linearization of \eqref{MFG_stat_eik_lap_discrete} with respect to $(U,M,\Lambda)$, along a direction $(W_U,W_M,W_\Lambda)$ and starting from an initial guess $(U^{(0)},M^{(0)},\Lambda^{(0)})$, we get the following Newton iterations for $k\ge 0$,
\begin{equation}\label{Newton}
\mathcal{J}[U^{(k)},M^{(k)},\Lambda^{(k)}]
\left(\begin{array}{c}
     W_U \\
     W_M\\
     W_\Lambda 
\end{array}\right)=-\mathcal{F}(U^{(k)},M^{(k)},\Lambda^{(k)})\,,
\end{equation}
with updates
$$
(U^{(k+1)},M^{(k+1)},\Lambda^{(k+1)})=(U^{(k)},M^{(k)},\Lambda^{(k)})+(W_U,W_M,W_\Lambda)\,,
$$
where, denoting by  $|\mathcal{G}|$  the number of nodes of $\mathcal{G}$, the map $\mathcal{F}:\R^{2|\mathcal{G}|+1}\to\R^{2|\mathcal{G}|+2}$ is defined as
\begin{equation}\label{Fresidual}
\mathcal{F}(U,M,\Lambda)=
\left(\begin{array}{c}
     -\varepsilon\D_\sharp U+\frac12 |D_\sharp U|^2-V_\sharp-F_\sharp(M)+\Lambda \\
     -\varepsilon\D_\sharp M-\diver_\sharp(M\,D_\sharp U)\\
      \int_\sharp U\\
    \int_\sharp M -1
\end{array}\right)\,,
\end{equation}
while the Jacobian matrix $\mathcal{J}$ is given by

$$
\mathcal{J}
[U,M,\Lambda]=
\left(
\begin{array}{c|c|c}
-\varepsilon\D_\sharp + D_\sharp U\cdot D_\sharp & -F^\prime_\sharp(M) & 1_\sharp \\
\hline
-\diver_\sharp(M^{(k)}\,D_\sharp\, \cdot) & 
-\varepsilon\D_\sharp -\diver_\sharp(\cdot\,D_\sharp U) & 0_\sharp\\
\hline
\int_\sharp & 0 & 0\\
\hline
0 & \int_\sharp & 0\\
\end{array}
\right)
$$
with $0_\sharp=(0,\dots,0)^T\in\R^{|\mathcal{G}|}$ and $1_\sharp=(1,\dots,1)^T\in\R^{|\mathcal{G}|}$.

Note that, for each $k\ge0$, the above linear system consists in $2|\mathcal{G}|+2$ equations in the $2|\mathcal{G}|+1$ unknowns $(W_U,W_M,W_\Lambda)$, and its solution is meant in a least-squares sense, see \cite{cc} for further details. By rewriting $(W_U,W_M,W_\Lambda)$ in terms of successive iterations, we readily end up with    
$$
	\begin{cases}
	 -\varepsilon\D_\sharp U^{(k+1)}+ D_\sharp U^{(k)}\cdot D_\sharp U^{(k+1)}  -F^\prime_\sharp(M^{(k)})(M^{(k+1)}-M^{(k)})+\Lambda^{(k+1)}\\=\frac12 |D_\sharp U^{(k)}|^2+V_\sharp+F_\sharp(M^{(k)})\,, & \\\\
	 -\diver_\sharp(M^{(k)}\,D_\sharp (U^{(k+1)}-U^{(k)}))-\varepsilon\D_\sharp M^{(k+1)}-\diver_\sharp(M^{(k+1)}\,D_\sharp U^{(k)})=0\,, & \\\\
	\int_\sharp M^{(k+1)}=1\,,\quad \int_\sharp U^{(k+1)} =0\,.
	\end{cases}
$$
Since both $U^{(k)}$ and $M^{(k)}$ are expected to converge, we can neglect, for $k$ large, the two terms  $F^\prime_\sharp(M^{(k)})(M^{(k+1)}-M^{(k)})$ and  $\diver_\sharp(M^{(k)}\,D_\sharp (U^{(k+1)}-U^{(k)}))$. This completely decouples the above system, and yields exactly the policy iteration algorithm by setting $Q^{(k)}=D_\sharp U^{(k)}$ at each iteration. Thus, we can reinterpret the policy iteration as a {\it quasi-Newton} method for the system \eqref{MFG_stat_eik_lap_discrete}, by dropping the two corresponding off-diagonal blocks in the Jacobian matrix $\mathcal{J}$. 

\section{Numerical simulations}\label{sec:test}
Here, we provide some details on the implementation of the policy iteration method. Then we present a comparison with the direct Newton method \eqref{Newton} for a stationary MFG system in dimension one, and a test in dimension two for the evolutive case. 

Concerning the stationary case, at each iteration $k$, the solution of the discrete FP equation in step (i) is obtained by the M-matrix regularization discussed in Theorem \ref{PI_dis}: 
starting from an initial guess $W^{(0)}\in\R^\mathcal{G}\setminus\{0\}$, with $W^{(0)}\ge 0$ and $\int_\sharp W^{(0)} =1$, we solve iteratively on $s\ge 0$ \begin{equation*}
	(\mu I +A(Q^{(k)})) W^{(s+1)} = \mu W^{(s)}\,,
\end{equation*}
namely a sequence of linear systems of size $|\mathcal G|\times |\mathcal G|$.
Note that this introduces an additional (inner) iteration in the policy iteration algorithm. Moreover, by rewriting each linear system in the form
$$
\frac{W^{(s+1)}-W^{(s)}}{\frac{1}{\mu}}=-A(Q^{(k)})) W^{(s)}\,,
$$
we can reinterpret the regularization as an implicit gradient descent scheme with step $\frac{1}{\mu}$ for finding the zeros of $A(Q^{(k)})$, via minimization of its associated quadratic form. It is clear that, as we increase $\mu$ to recover the M-matrix property, we dramatically loose the advantage of an implicit scheme, slowing down the convergence of $W^{(s)}$.
In practice we can tune the parameter $\mu$ for the specific test, and perform only a fixed number of inner iterations. 

For the remaining steps of the policy iteration algorithm, we observe that step (ii) corresponds to the solution of a square linear system of size $(|\mathcal{G}|+1)\times (|\mathcal{G}|+1)$ in the unknowns  $(U^{(k)},\Lambda^{(k)})$, while the policy update in step (iii) is explicit due to the particular choice for the Hamiltonian. In the general case, according to \eqref{eq:update_policy_stat}, a point-wise optimization on $\mathcal{G}$ is needed to obtain the new policy. 

On the other hand, each iteration in the direct Newton method \eqref{Newton} requires the solution of a system of size $(2|\mathcal{G}|+2)\times (2|\mathcal{G}|+1)$ in a least-squares sense. Both algorithms are implemented in $C$ language, 
employing the free library SuiteSparseQR \cite{SPQR} for solving the linear systems via $QR$ factorization. To check convergence, given a tolerance $\tau>0$, we rely on 
the $2$-norm of the residual for the discrete nonlinear system \eqref{Fresidual}, requiring $\|\mathcal{F}(U^{(k)},M^{(k)},\Lambda^{(k)})\|_2<\tau$.  

In the following test, we  set the problem in dimension $d=1$, with $\tau=10^{-8}$, $\varepsilon=0.3$, $V(x)=\sin(2\pi x)+\cos(4\pi x)$ and $F(m)=m^2$. In particular, the choice of the coupling cost satisfies the monotonicity assumption \eqref{unique}, ensuring uniqueness of solutions for the MFG system. Moreover, we set the initial guess for the Newton method as $U^{(0)}\equiv 0$, $M^{(0)}\equiv 1$ on $\mathcal{G}$ and  $\Lambda^{(0)}=0$, while we take the initial policy $Q^{(0)}\equiv(0,0)$ on $\mathcal{G}$ for the policy iteration algorithm. Finally, for the inner M-matrix iterations, we set $\mu=10^{-3}$ and $s=1$, with $W^{(0)}\equiv 1$ for $k=0$ and $W^{(0)}=M^{(k-1)}$ for $k\ge 1$.   

Figure \ref{test1} shows the solution computed by the policy iteration algorithm on a grid with $|\mathcal{G}|=200$ nodes, while in Figure \ref{test-comp} we compare the performace of the two methods. \begin{figure}[h!]
 \begin{center}
 \begin{tabular}{cc}
\includegraphics[width=.45\textwidth]{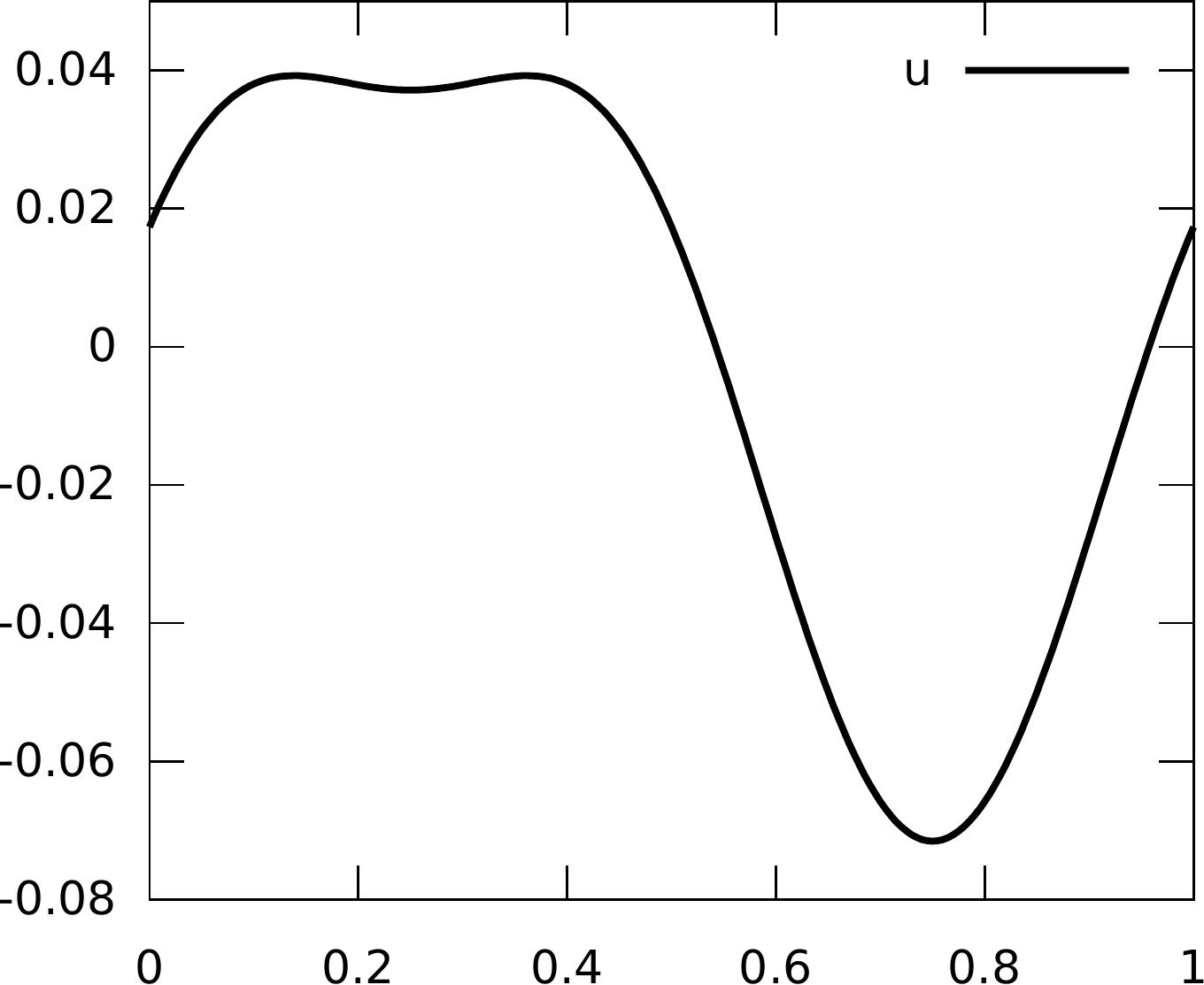} &
\includegraphics[width=.45\textwidth]{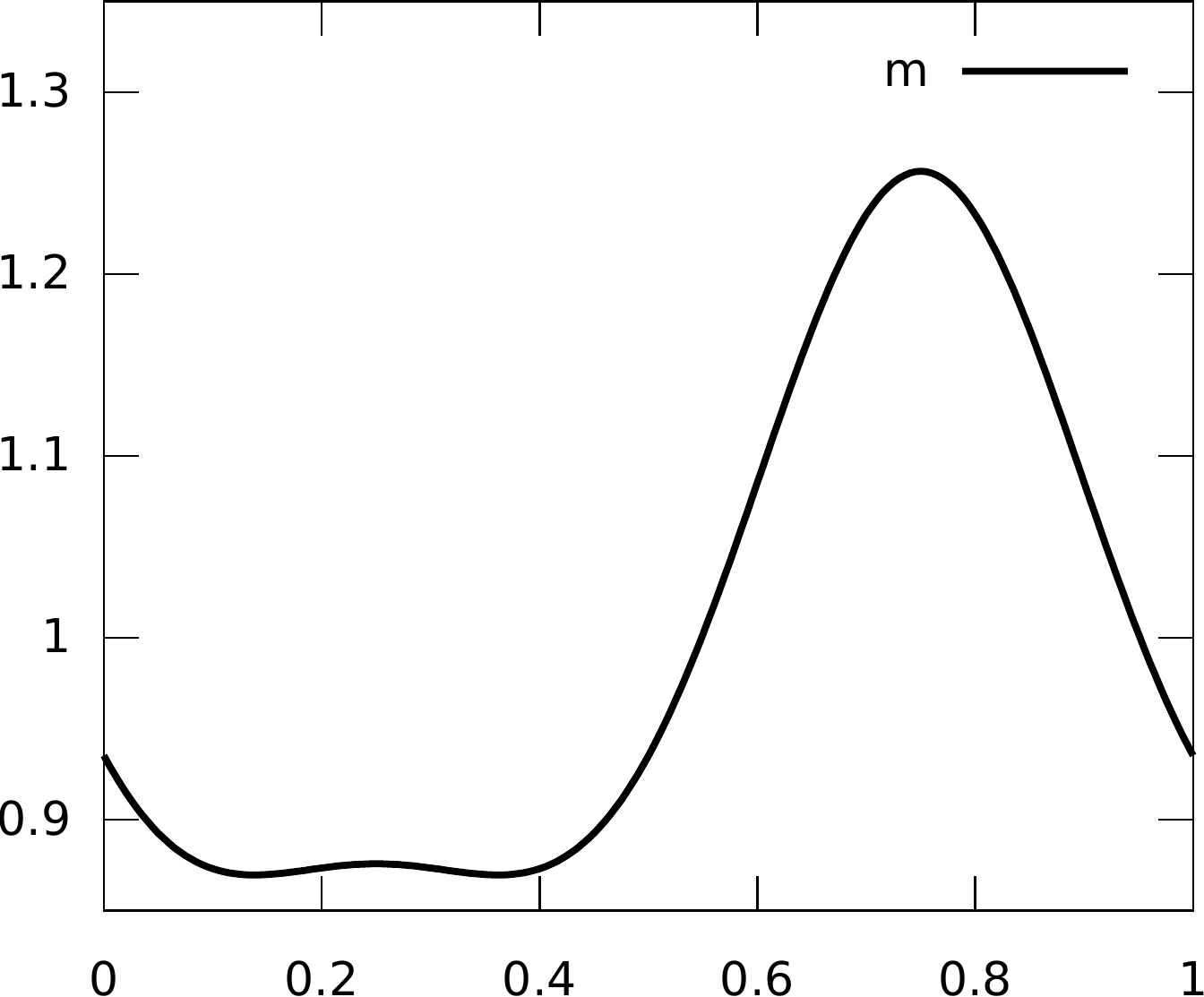} \\
(a)&(b)
\end{tabular}
\end{center}
\caption{Policy iteration solution for the stationary MFG system,  (a) the corrector $u$ and  (b) the density $m$.}\label{test1}
\end{figure}
\begin{figure}[h!]
 \begin{center}
 \begin{tabular}{cc}
\includegraphics[width=.4\textwidth]{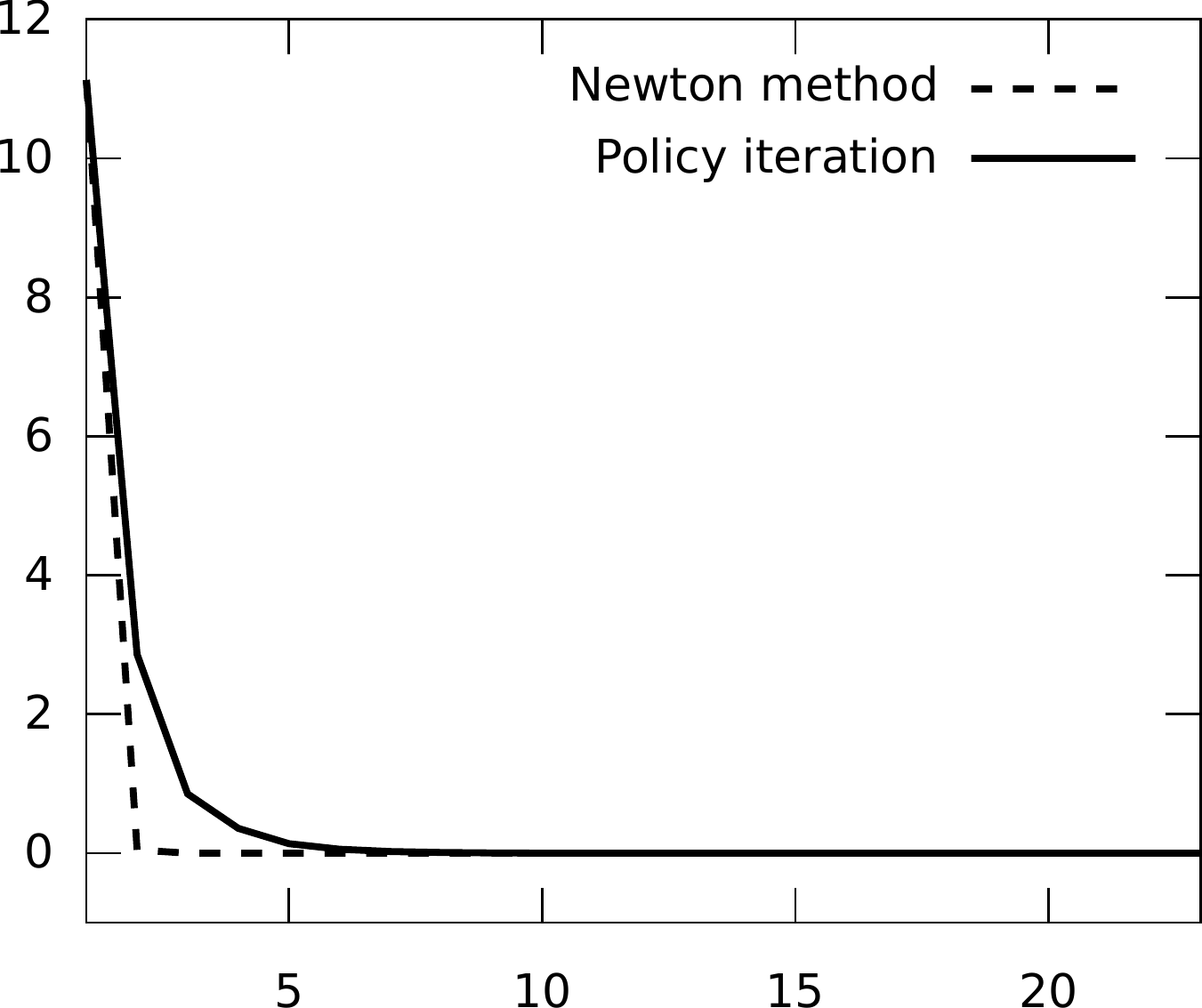} &
\includegraphics[width=.4\textwidth]{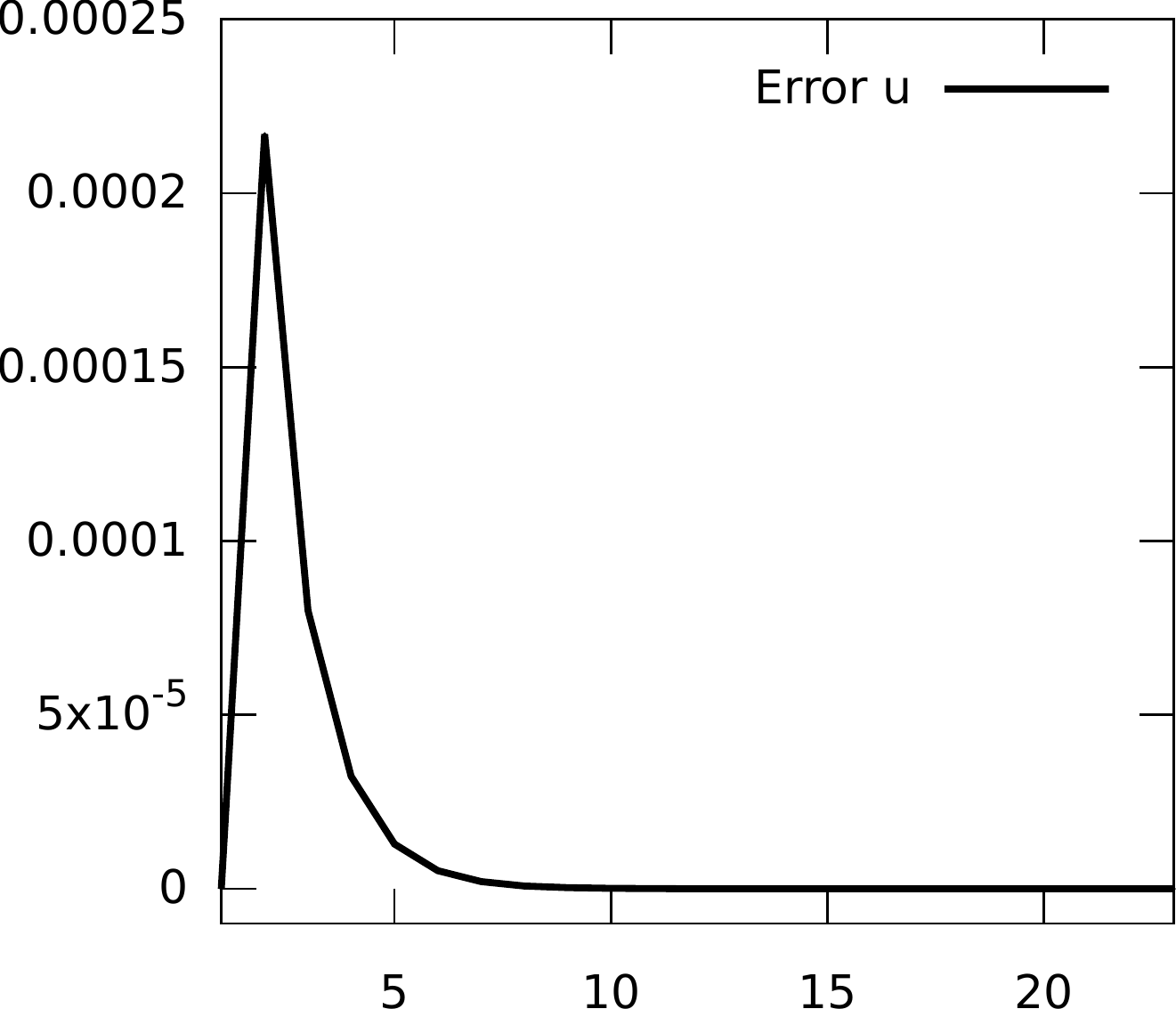} \\
(a)&(b)\\
\includegraphics[width=.4\textwidth]{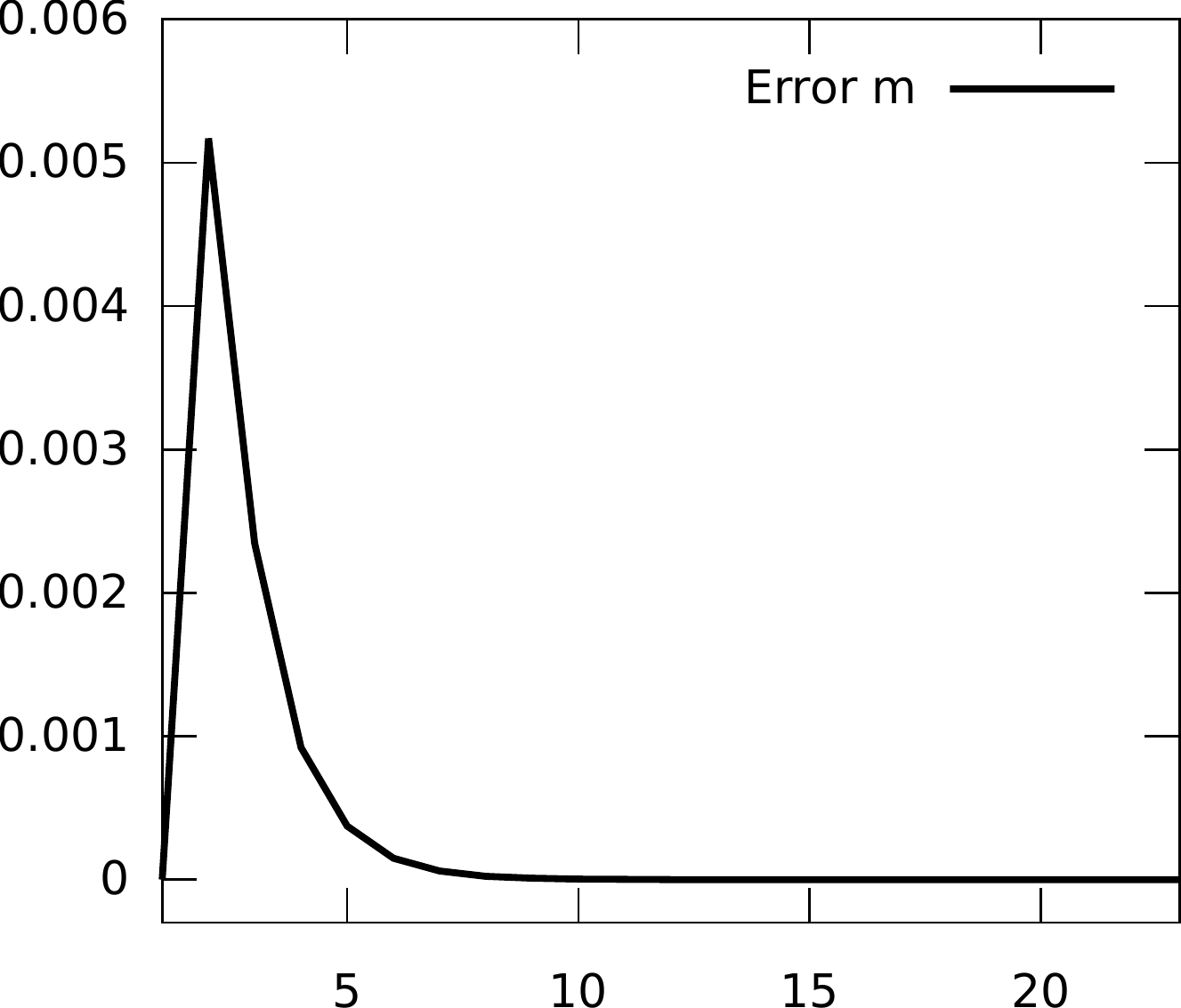}&
\includegraphics[width=.4\textwidth]{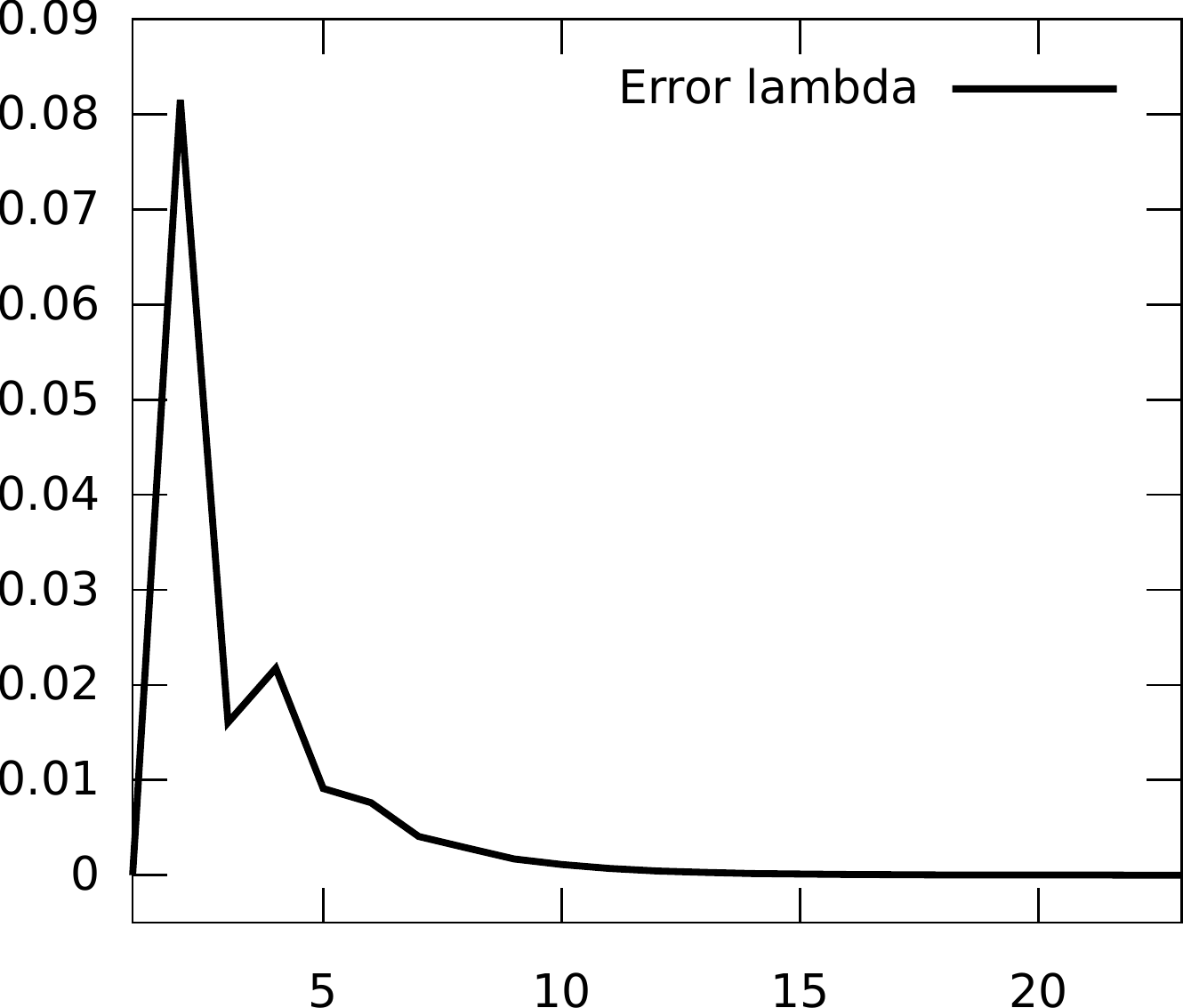} \\
(c)&(d)
\end{tabular}
\end{center}
\caption{Policy iteration vs Newton method, (a) MFG system residuals and (b-c-d) differences in the solutions $u$, $m$, $\lambda$.}\label{test-comp}
\end{figure} 
More precisely, in Figure \ref{test-comp}a, we show the residuals of the two methods, against the number of iterations needed to reach the given tolerance $\tau$. The Newton method converges in just $5$ iterations, while the policy iteration requires $24$ iterations. Similarly, in Figure \ref{test-comp}b-c-d we show the differences between the solutions of the two methods in the discrete $L^2$ norm, against the number of iterations. Due to the particular choice of the initial guess, at the first iteration the two methods compute the same solution, but the policy iteration algorithm requires more iterations to reach the same accuracy for the residual. This is clearly expected, since the Newton method employs the descent direction associated to the full Jacobian matrix $\mathcal{J}$. Nevertheless, as reported in Table \ref{table-policyvsdirect}, the policy iteration exhibits a better performance as the number of grid nodes increases, due to the reduced size of the corresponding linear systems (see the averaged CPU times per iteration).
\begin{table}[!h]
 \centering
  \begin{tabular}{|c|c|c|c|c|}
    \hline
     & $|\mathcal{G}|$ & Its & Av.CPU/It (secs) & Total CPU (secs)\\
     \hline\hline
   NM & 200 & 5 & 0.006 & 0.034 \\
    \hline
    PI & 200 & 24 & 0.003 & 0.079 \\
    \hline\hline
    NM & 500 & 5 & 0.037 & 0.189       \\
    \hline
    PI & 500 & 25 & 0.009 & 0.247 \\
    \hline\hline
    NM & 1000 & 5 & 0.173 & 0.865    \\
    \hline
    PI & 1000 & 25 & 0.036 & 0.917   \\
    \hline\hline
    NM & 2000 & 5 &  0.973 & 4.869    \\
    \hline
    PI & 2000 & 25 & 0.241 & 6.039   \\
    \hline\hline
    NM & 5000 & 5 &  13.662 & 68.313  \\
    \hline
    PI & 5000 & 25 & 1.724 & 43.115   \\
    \hline\hline
    NM & 10000 & 5 &  123.769 & 618.845   \\
    \hline
    PI & 10000 & 25 & 7.917 & 197.949   \\
    \hline
    \end{tabular}
  \caption{Policy iteration (PI) vs Newton method (NM) under grid refinement, number of iterations, averaged CPU times per iteration, and total CPU times.}\label{table-policyvsdirect}
\end{table}
 We must observe that the comparison is not truly fair, since the update step for the policy iteration is explicit in this example, with a negligible computational cost. However, in the general case, we expect that the relevant speed-up of the proposed algorithm on large grids can compensate the efforts for the optimization process \eqref{eq:update_policy_stat}, since it is a point-wise procedure that can be completely parallelized. 

Now, let us consider the evolutive MFG system \eqref{MFG}, again in the special case of the Eikonal-diffusion HJB equation, but in dimension $d=2$. Spatial discretization is performed in both dimensions as in the one dimensional case, while, for time discretization, we employ an implicit Euler method for both the time-forward FP equation and the time-backward HJB equation. To this end, we introduce a uniform grid on the interval $[0,T]$ with $N+1$ nodes $t_n=n \,dt$, for $n=0,\dots, N$, and time step $dt=T/N$. Then, we denote by $U_n, M_n$ and $Q_n$ the vectors on $\mathcal{G}$ approximating respectively the solution and the policy at time $t_n$. In particular, we set on $\mathcal{G}$ the initial condition $M_0=m_0(\cdot)$ and the final condition $U_N=u_T(\cdot)$. The policy iteration algorithm for the fully discretized system is the following:
 given an initial guess $Q^{(0)}_n:\mathcal{G}\to\R^{2d}$ for $n=0,\dots, N$, initial and final data $M_0,\,U_N:\mathcal{G}\to\R$, iterate on $k\ge 0$ up to convergence, 

    \begin{itemize}
    	\item[(i)]  Solve for $n=0,\dots, N-1$ on $\mathcal{G}$
    	$$
    	\left\{
    	\begin{array}{l}
    	 M^{(k)}_{n+1}-dt\left(\varepsilon\D_\sharp M^{(k)}_{n+1}+\diver_\sharp(M^{(k)}_{n+1}\,Q^{(k)}_{n+1})\right)=M^{(k)}_{n}\\
    	 M^{(k)}_{0}=M_{0}    
    	\end{array}
    	\right.
    	$$
    	\item[(ii)] Solve for $n=N-1,\dots, 0$ on $\mathcal{G}$
    	$$
    	    	\left\{
    	\begin{array}{l}
    U_{n}^{(k)}- dt\left(	 \varepsilon\D_\sharp U^{(k)}_{n}-Q^{(k)}_{n,\pm}\cdot D_\sharp U^{(k)}_{n}\right)\\
    \hskip 19pt= U_{n+1}+dt\left(\frac12 |Q^{(k)}_{n+1,\pm}|^2+V_\sharp+F_\sharp(M^{(k)}_{n+1})\right)
\\
    	 U^{(k)}_{N}=U_{N}    
    	\end{array}
    	\right.
      	$$
    	\item[(iii)] Update the policy
    	$Q^{(k+1)}_n=D_\sharp U^{(k)}_n$ on $\mathcal{G}$ for $n=0,\dots, N$, and set $k\leftarrow k+1$.
    \end{itemize}
Note that each iteration of the algorithm now requires the solution of $2N$ linear systems of size $|\mathcal{G}|\times|\mathcal{G}|$.

 In the following test, we choose a number of nodes $I=50$ for each space dimension and $N=100$ nodes in time, corresponding to $200$ linear systems of size $2500\times 2500$ per iteration.   
We set the final time $T=1$, the diffusion coefficient $\varepsilon=0.3$, the coupling cost $F(m)=m^2$ and the potential $V(x_1,x_2)=-|\sin(2\pi x_1)\sin(2\pi x_2)|$. Moreover, to check convergence, we rely on the discrete $L^2$ squared distance between policies at successive iterations, i.e. we stop the algorithm when  $\displaystyle\max_n\int_\sharp|Q^{(k+1)}_n-Q^{(k)}_n|^2<\tau$, setting the tolerance $\tau=10^{-8}$. Finally, we take the initial policy 
$Q^{(0)}_n\equiv(0,0,0,0)$ 
on $\mathcal{G}$ for $n=0,\dots, N$, while we define the initial and final data $M_0$ and $U_N$ approximating on $\mathcal{G}$ the functions 
$m_0(x_1,x_2)=-u_T(x_1,x_2)=C\exp\left\{-40[(x_1-\frac{1}{2})^2+(x_2-\frac{1}{2})^2]\right\}$,
namely two Gaussian with opposite signs centered at the point 
$(\frac12,\frac12)$, 
with $C>0$ such that $\int_{\mathbb{T}^2}m_0(x)dx=1$. \\
The algorithm requires $58$ iterations to reach convergence up to $\tau$, with an averaged CPU time per iteration of $7.3$ seconds, and a total CPU time of $423$ seconds. In Figure \ref{test-time}, we report some relevant frames of the time evolution, by plotting, for $n$ fixed, the solution density $M_n$ in gray scales, and superimposing the optimal dynamics for the FP equation, which is obtained by merging the two-sided components of $Q_n$, namely  $(Q_{n,L}^1+Q_{n,R}^1,Q_{n,L}^2+Q_{n,R}^2)$. 
\begin{figure}[h!]
 \begin{center}
 \begin{tabular}{ccc}
\includegraphics[width=.3\textwidth]{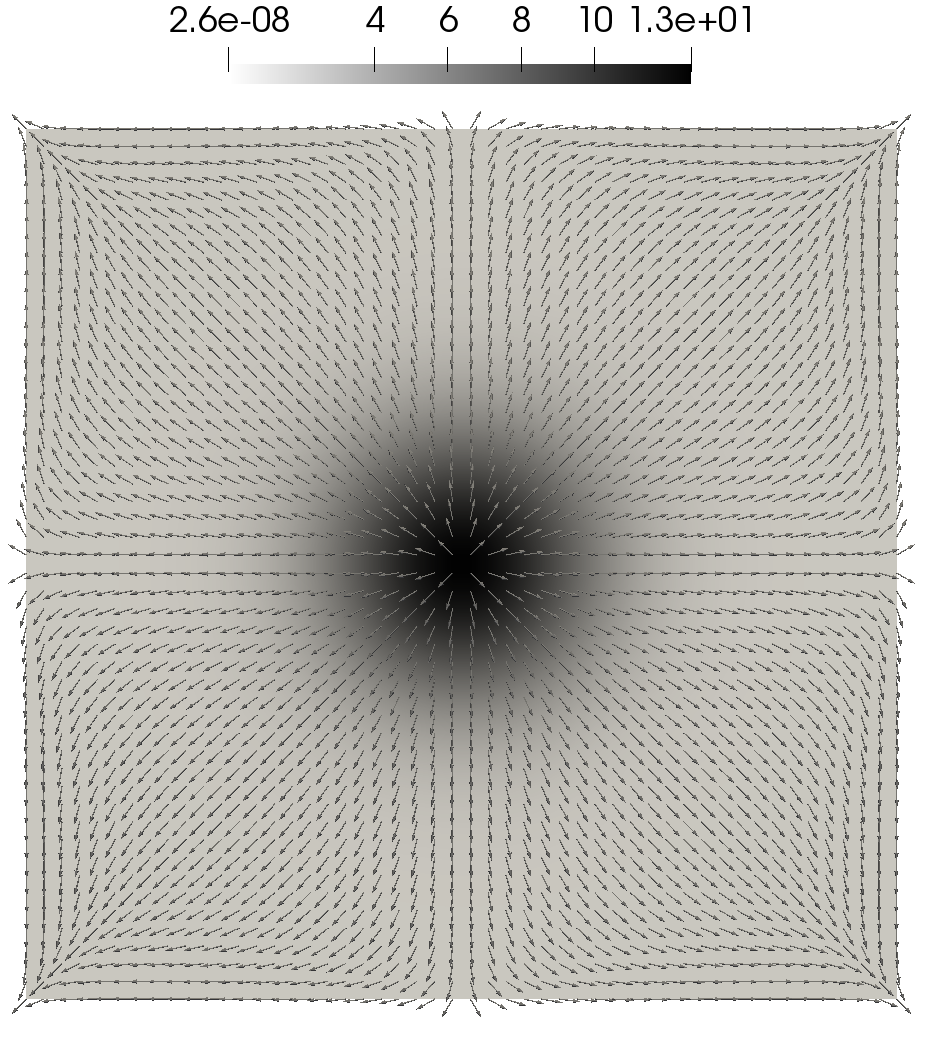} &
\includegraphics[width=.3\textwidth]{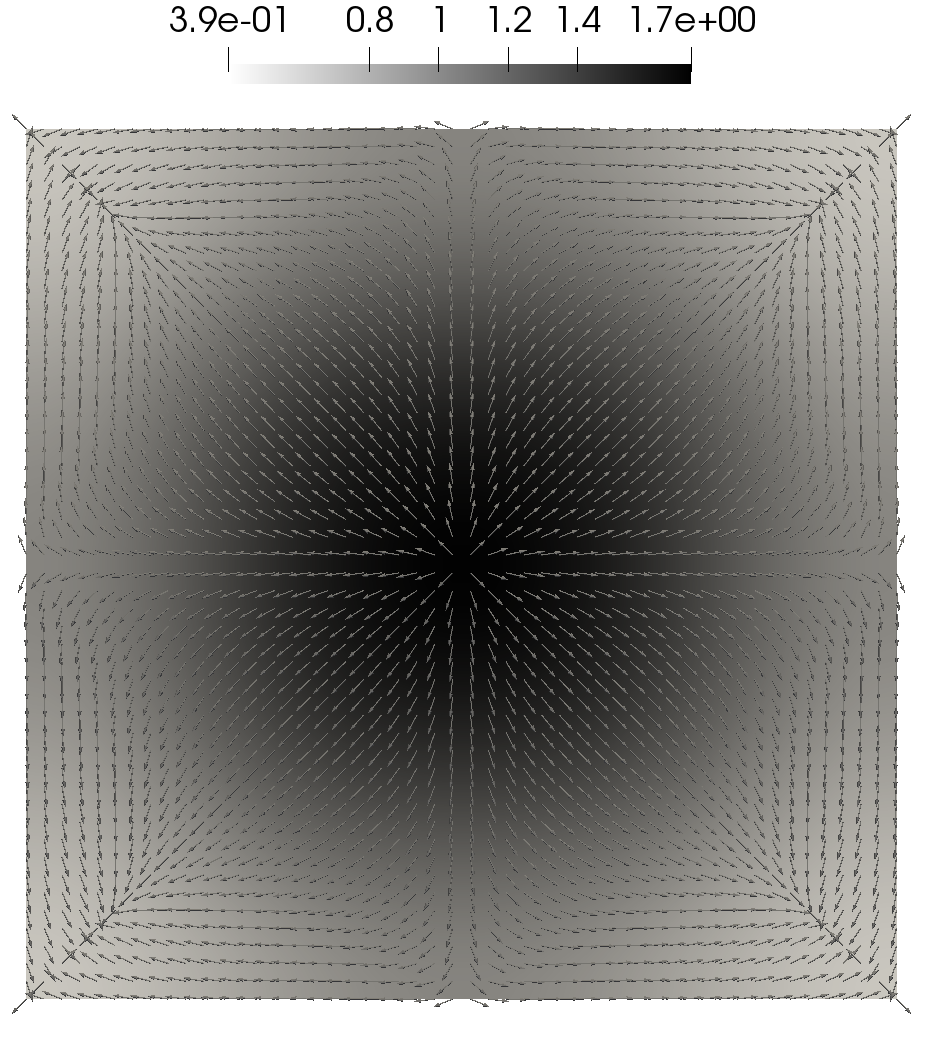} &
\includegraphics[width=.3\textwidth]{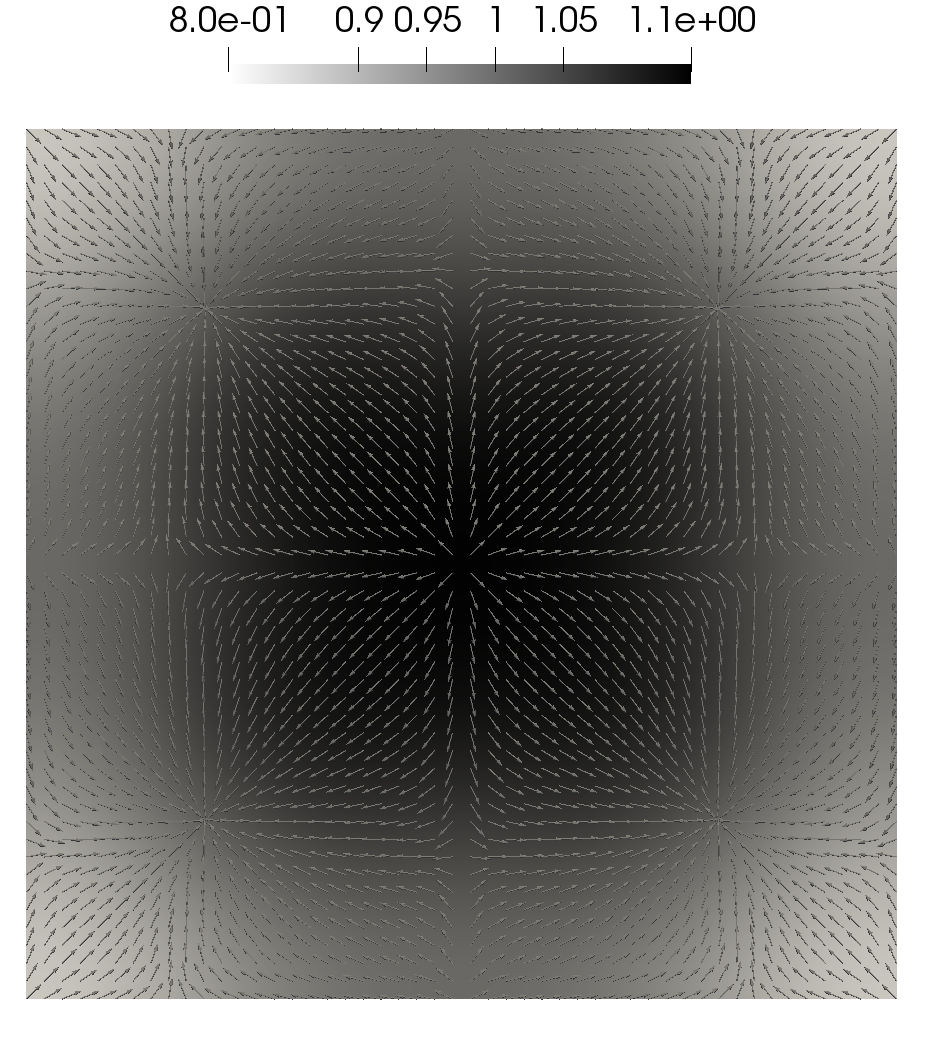} \\
$t=0$ & $t=0.1$ &$t=0.2$\\\\
\includegraphics[width=.3\textwidth]{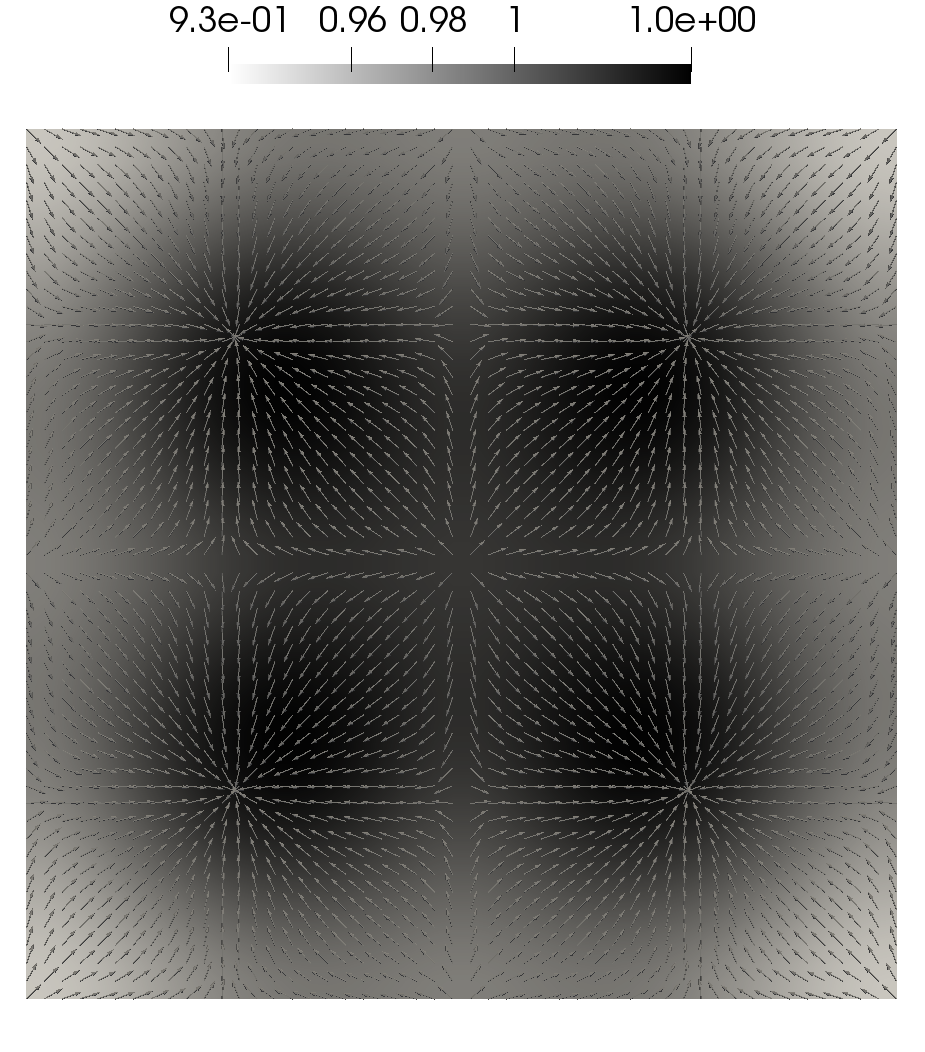} &
\includegraphics[width=.3\textwidth]{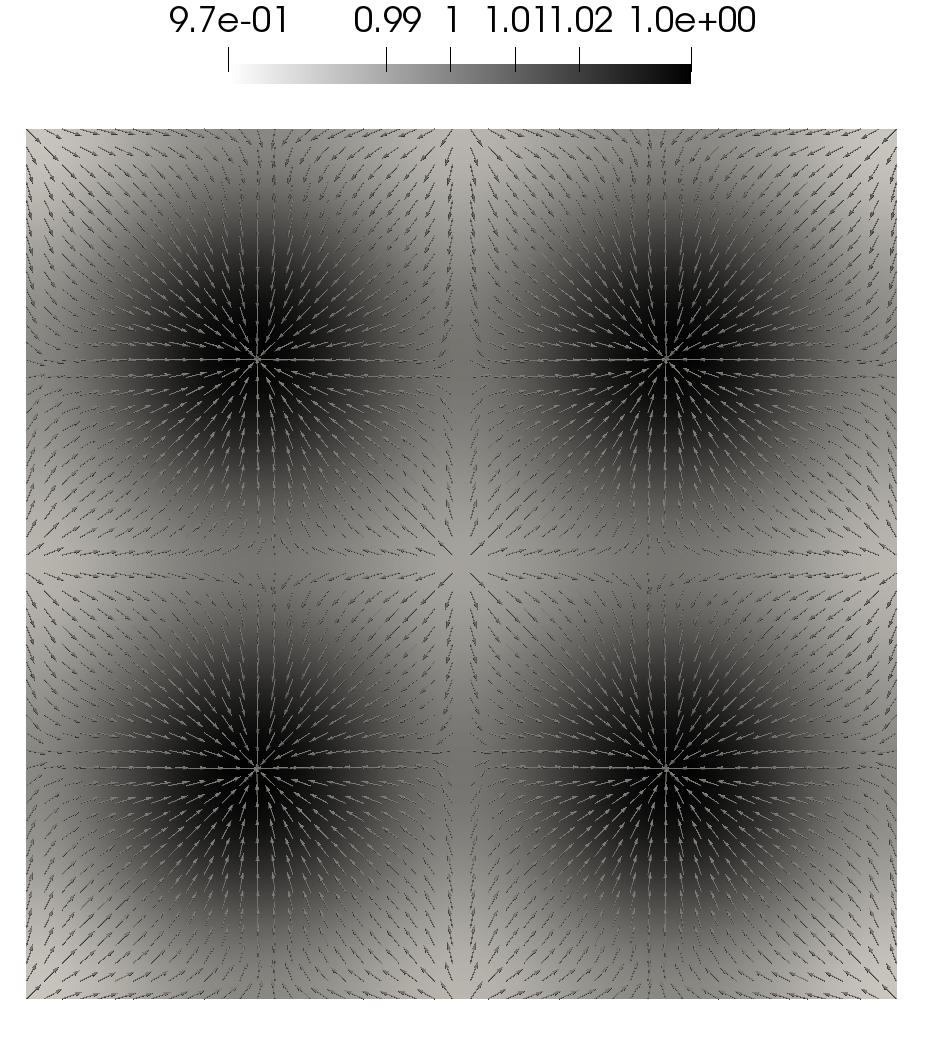} &
\includegraphics[width=.3\textwidth]{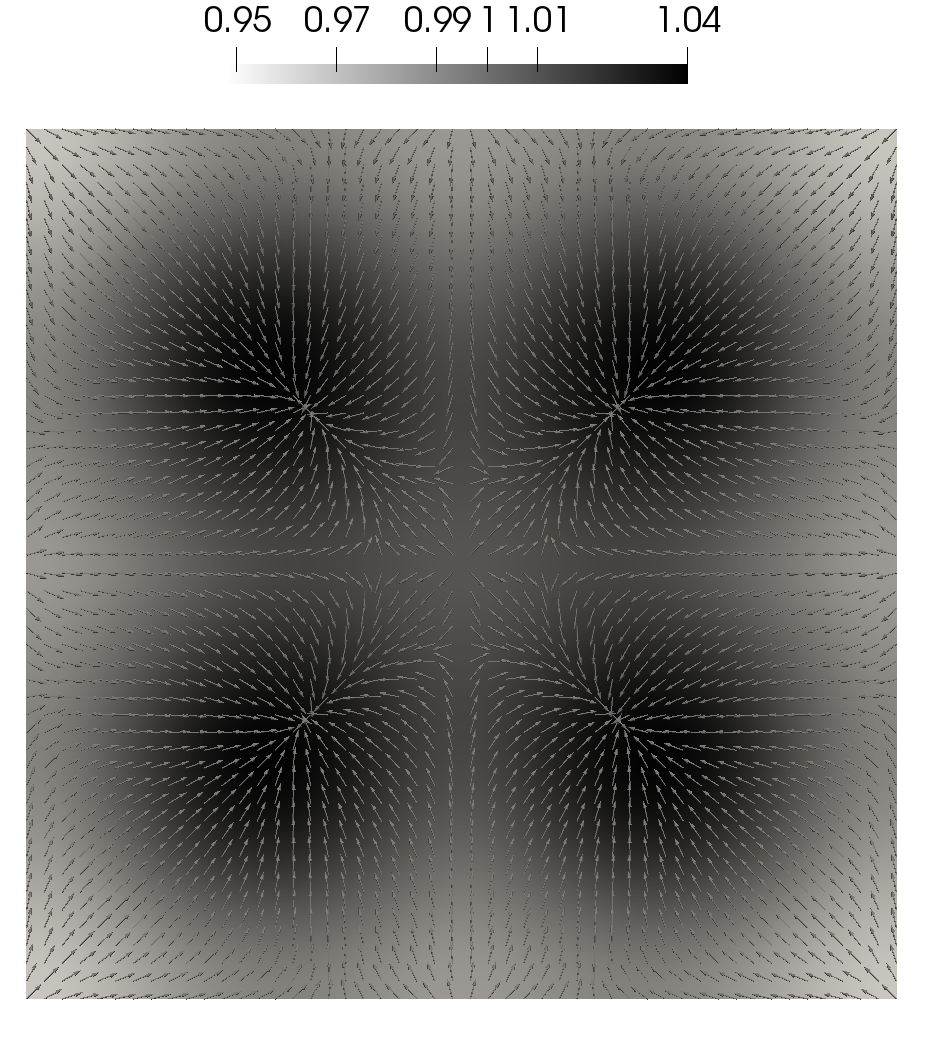} \\
$t=0.3$ & $t=0.7$ &$t=0.8$\\\\
\includegraphics[width=.3\textwidth]{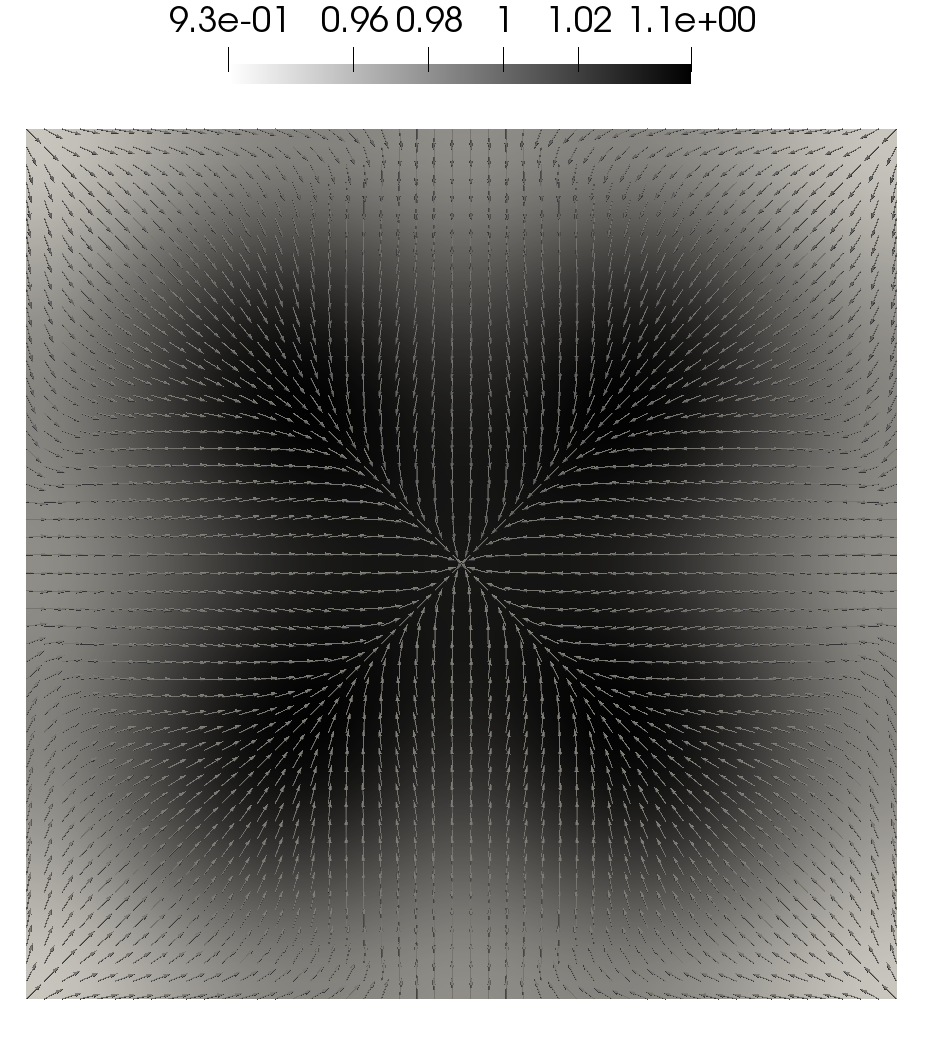} &
\includegraphics[width=.3\textwidth]{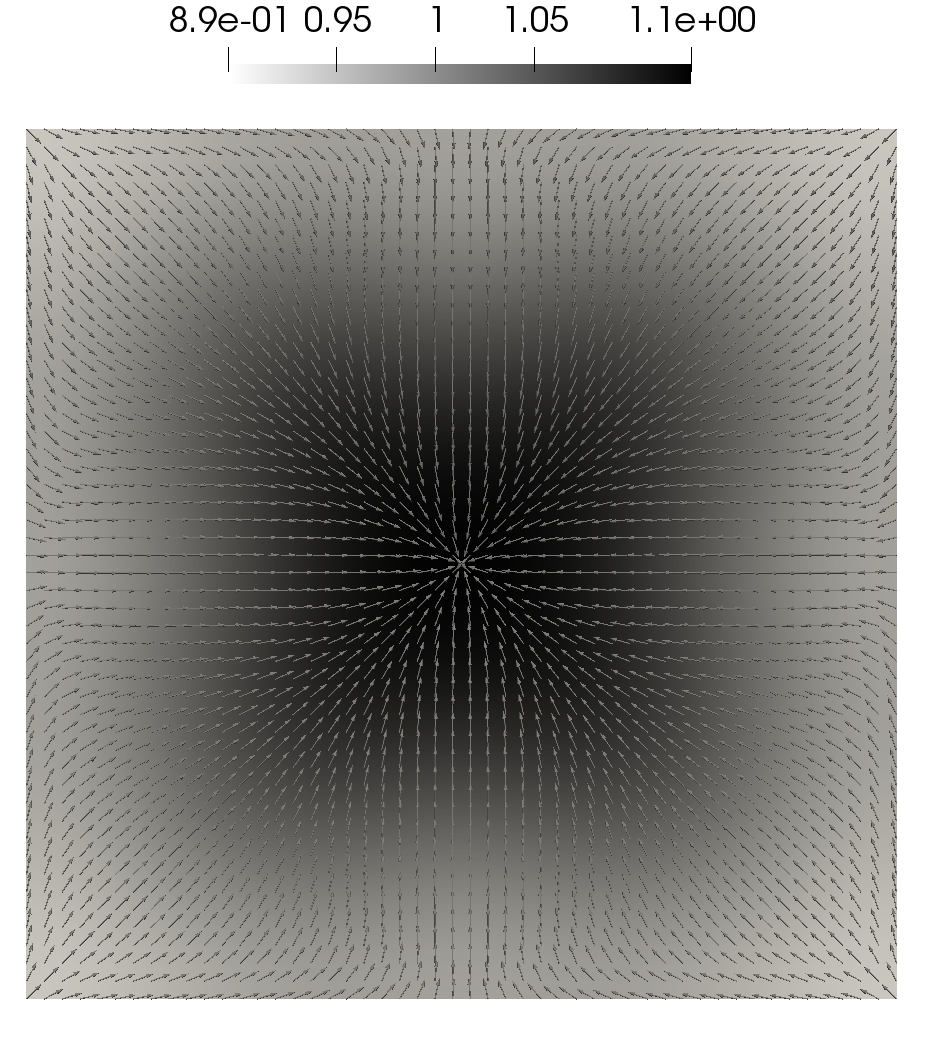} &
\includegraphics[width=.3\textwidth]{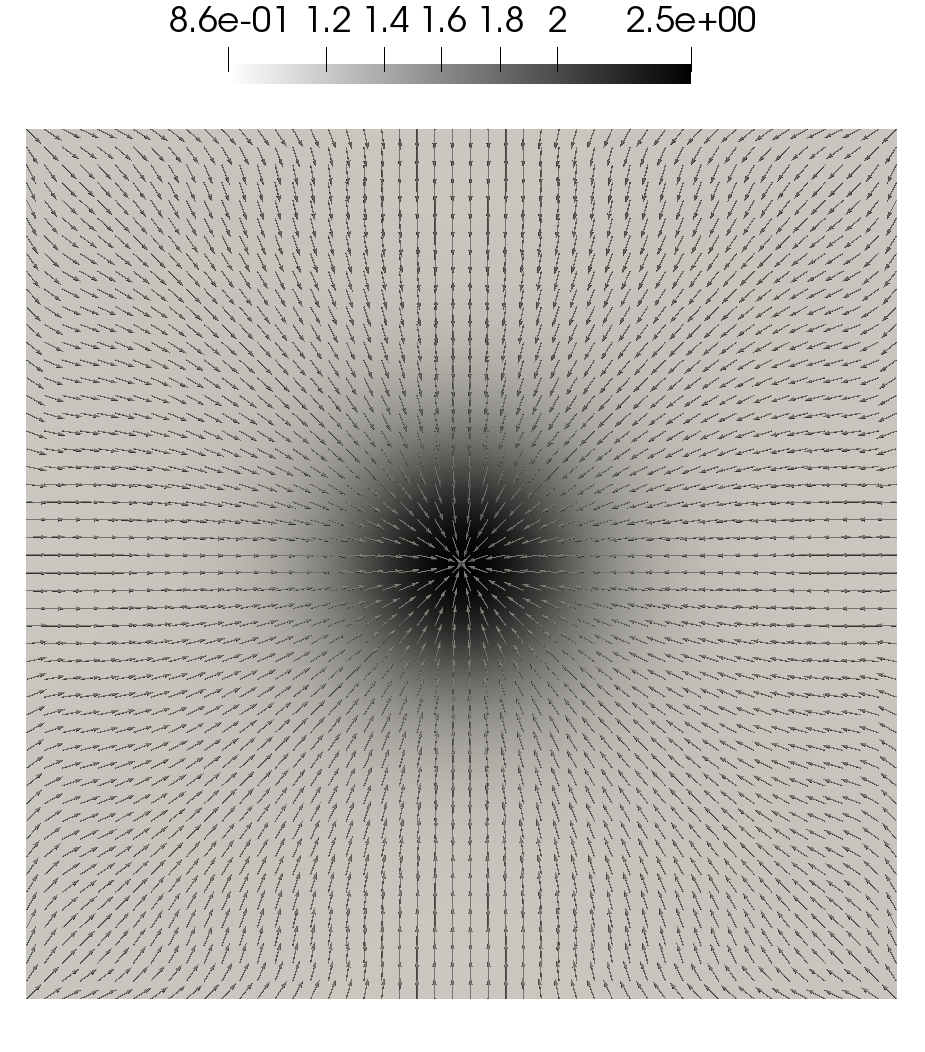} \\
$t=0.85$ & $t=0.9$ &$t=1$
\end{tabular}
\end{center}
\caption{Solution of the evolutive MFG system at different times, mass density in gray scales and optimal dynamics.}\label{test-time}
\end{figure}
We remark that, by definition, the absolute minimum of the potential $V$ is achieved at the points $(\frac14,\frac14)$, $(\frac34,\frac14)$, $(\frac14,\frac34)$, $(\frac34,\frac34)$. We observe that the optimal dynamics readily splits the density symmetrically in four parts, pushing them to concentrate around these minimizers, while, in the final part of the time interval $[0,T]$, it forces the density to merge again and concentrate exactly around the point $(1/2,1/2)$ (i.e. the absolute minimizer of $u_T$), in order to to satisfy the final condition for the HJB equation.
This configuration corresponds to the so called {\it turnpike} phenomenon \cite{pt}. Roughly speaking, it turns out that the solution of the evolutive problem corresponds to approach the solution of the corresponding stationary ergodic problem, standing on this equilibrium as long as possible before moving again towards $u_T$.  
\bigskip

{\bf Acknowledgements.} 
We warmly thank the anonymous referees for their careful reading of the manuscript and the precious comments that improved the presentation. The third-named author has been partially supported by 
the Fondazione CaRiPaRo
Project ``Nonlinear Partial Differential Equations:
Asymptotic Problems and Mean-Field Games". 

\small

\medskip
\begin{flushright}
\noindent \verb"cacace@mat.uniroma3.it"\\
	Dipartimento di Matematica e Fisica\\
	Universit\`{a} degli Studi Roma Tre\\
	Largo S. L. Murialdo 1, 00146   Roma (Italy)\\
\smallskip	
\noindent \verb"fabio.camilli@uniroma1.it"\\
	SBAI, Sapienza Universit\`{a} di Roma\\
	via A.Scarpa 14, 00161 Roma (Italy)	\\
\smallskip	
\noindent \verb"alessandro.goffi@math.unipd.it"\\
Dipartimento di Matematica\\
Universit\`a di Padova\\
via Trieste 63, 35121 Padova (Italy)

\end{flushright}


\begin{thebibliography}{99}

	\bibitem{acd}
	Achdou, Y.; Capuzzo Dolcetta, I. Mean Field Games: numerical methods, SIAM J. Numer. Anal., 48 (2010), pp. 1136-1162.
	
	\bibitem{accd}
	Achdou, Y.; Camilli, F.; Capuzzo Dolcetta, I. Mean field games: convergence of a finite
	difference method. SIAM J. Numer. Anal. 51 (2013), no. 5, 2585-2612.
	
	\bibitem{al}
	 Achdou, Y.; Lauriere, M. On the system of partial differential equations arising in mean field type control. Discrete Contin. Dyn. Syst. 35 (2015), no. 9, 3879-3900.
	
	\bibitem{al1}
	Achdou, Y.; Lauriere, M. Mean Field Games and applications: numerical aspects. In {\it Mean field games}, 249–307, Lecture Notes in Math., 2281,  Springer, Cham, 2020.
	
	\bibitem{alla}
	Alla, A.; Falcone, M.; Kalise, D. An efficient policy iteration algorithm for dynamic programming equations. SIAM J. Sci. Comput. 37 (2015), no. 1, A181-A200. 
	
	\bibitem{bf}
	Bardi, M.; Feleqi, E. Nonlinear elliptic systems and mean-field games.  
	NoDEA Nonlinear Differential Equations Appl. 23 (2016), no. 4, Art. 44, 32 pp.

    \bibitem{bp}
    Bardi, M.; Priuli, F. Linear-quadratic N-person and mean-field games
    with ergodic cost, SIAM J. Control Optim. 52 (2014), no. 5, 3022-3052.
    
	\bibitem{b}
	Bellman, R. {\it Dynamic Programming}. Princeton Univ. Press, Princeton, 1957.
	

     
	\bibitem{BCCS}
	Bianchini, S.; Colombo, M.; Crippa, G.; Spinolo, L. V. Optimality of integrability estimates for advection-diffusion equations. NoDEA Nonlinear Differential Equations Appl. 24 (2017), no. 4, Paper No. 33, 19 pp.
	

	\bibitem{bmz}
	Bokanowski, O.; Maroso, S.; Zidani, H. Some convergence results for Howard's algorithm. SIAM J. Numer. Anal. 47 (2009), no. 4, 3001-3026. 
	 			 
	 \bibitem{bcarda}
     Briani, A.; Cardaliaguet, P.; Stable solutions in potential mean field game systems. NoDEA Nonlinear Differential Equations Appl. 25 (2018), no. 1.
     
	\bibitem{bks}
	Brice\~{n}o-Arias, L. M.; Kalise, D.; Silva, F. J. Proximal methods for stationary mean
	field games with local couplings. SIAM J. Control Optim. 56 (2018), no. 2, 801-836.



	\bibitem{cc}
	Cacace, S.;  Camilli, F. A generalized Newton method for homogenization
	of Hamilton-Jacobi equations, SIAM J. Sci. Comput. 38 (2016), no. 6,
	A3589-A3617.
	

	 \bibitem{ch}
     Cardaliaguet, P.; Hadikhanloo, S. Learning in mean field games: the fictitious play. ESAIM Control Optim. Calc. Var. 23 (2017), no. 2, 569-591.
          
     \bibitem{nhm}
      Cardaliaguet, P.; Lasry, J.-M.; Lions, P.-L.; Porretta, A. Long time average of mean field games. Netw. Heterog. Media 7 (2012), no. 2, 279--301.
      
	\bibitem{cs}
	Carlini, E.; Silva, F. A semi-Lagrangian scheme for a degenerate second order mean
	field game system. Discrete Contin. Dyn. Syst. 35 (2015), no. 9, 4269-4292.
	
	\bibitem{cdl}
	Carmona, R.; Delarue, F.; Lacker, D. Mean field games of timing and models for bank runs. Appl. Math. Optim. 76 (2017), no. 1, 217--260.

	
	
	\bibitem{CG1}
	Cirant, M.; Goffi, A. On the existence and uniqueness of solutions to time-dependent fractional MFG. SIAM J. Math. Anal. 51 (2019), no. 2, 913-954.
	
	\bibitem{CG2}
	Cirant, M.; Goffi, A. Lipschitz regularity for viscous Hamilton-Jacobi equations with $L^p$ terms. Ann. Inst. H. Poincar\'e Anal. Non Lin\'eaire 37 (2020), no. 4, 757-784.
	
	
	\bibitem{CG4}
	Cirant, M.; Goffi, A. On the problem of maximal $L^q$-regularity for viscous Hamilton-Jacobi equations. Arch. Rat. Mech. Anal. 240 (2021), 1521-1534.
	

 \bibitem{SPQR}
Davis, T.; SuiteSparse,
\texttt{http://faculty.cse.tamu.edu/davis/suitesparse.html}



	\bibitem{fl}
	Fleming, W.~H. Some Markovian optimization problems. J. Math. Mech. 12 (1963), 131-140.	
	
	\bibitem{gnp}
	Gomes, D. A.; Nurbekyan, L.; Pimentel, E. A. {\it Economic models and mean-field games theory}. IMPA Mathematical Publications,  Instituto Nacional de Matemática Pura e Aplicada (IMPA), Rio de Janeiro, 2015. iv+127 pp.
	
 
	\bibitem{h}
	Howard, R. {\it Dynamic Programming and Markov Processes}. MIT Press, Cambridge, 1960.
	 
	\bibitem{hcm}
	Huang,  M.; Caines,  P.~E.; Malhame,  R.~P.  Large-population cost-coupled LQG problems with non uniform agents: Individual-mass behaviour and decentralized $\epsilon$-Nash equilibria. IEEE Transactions on Automatic  Control, 52 (2007), 1560-1571.
	
	\bibitem{kss}
	Kerimkulov, B.; \v{S}i\v{s}ka, D.;  Szpruch, L.
	Exponential convergence and stability of Howards's policy improvement algorithm for controlled diffusions, SIAM J. Control Optim. 53 (2020), 1314--1340.
	
	\bibitem{LSU}
	 Ladyzenskaja, O.~A.; Solonnikov,  V.~A.; Ural'ceva, N.~N.
     {\it Linear and quasilinear equations of parabolic type}.
	 Translated from the Russian by S. Smith. Translations of Mathematical
	 Monographs, Vol. 23. American Mathematical Society, Providence, R.I., 1968.
	 
	 \bibitem{ll}
	 Lasry,  J.-M.;  Lions, P.-L. Mean field games. Jpn. J. Math. 2(2007), 229--260.
	 
	\bibitem{lions85}
	Lions, P.-L.   Quelques remarques sur les problemes elliptiques quasilin\'aires du second ordre. J. Analyse Math. 45 (1985), 234-254.

	\bibitem{Lun}
   Lunardi, A.\newblock  {\it Interpolation theory}. Vol. 16,  Appunti della Scuola Normale
  Superiore di Pisa (Nuova Serie), 2018.


	\bibitem{MPR}
    Metafune, G.; Pallara, D.;  Rhandi A. Global properties of transition probabilities of singular diffusions.
    Teor. Veroyatn. Primen., 54 (2009), 116--148.

   \bibitem{pt}
   Porretta, A. On the turnpike property for mean field games. Minimax Theory Appl. 3 (2018), no. 2, 285-312.


	
	
	\bibitem{pu1}
	Puterman, M.~L. On the convergence of policy iteration for controlled diffusions. J. Optim. Theory Appl. 33 (1981), no. 1, 137-144.
	
	
	\bibitem{pu2}
	Puterman, M.~L.  Optimal control of diffusion processes with reflection. J. Optim. Theory Appl. 22 (1977), no. 1, 103-116.
	
	\bibitem{pb}
	Puterman, M.~L.; Brumelle, S.~L.
	On the convergence of policy iteration in stationary dynamic programming. 
	Math. Oper. Res. 4 (1979), 60-69.
	


\bibitem{santos}
 Santos, M. S.; Rust, J. Convergence properties of policy iteration. SIAM J. Control Optim. 42 (2004), no. 6, 2094-2115.
 
\bibitem{ST}
Schmeisser, H.-J. ; Triebel, H.
 {\it Topics in {F}ourier analysis and function spaces}.
 A Wiley-Interscience Publication. John Wiley \& Sons, Ltd., Chichester, 1987.

	
	
\end{thebibliography}
 \end{document}